\documentclass[12pt]{amsart}
\usepackage{amsfonts,amsmath,amscd}
\usepackage{epsfig,graphics}
\usepackage{amssymb}
\usepackage{amscd}
\usepackage{hyperref}
\usepackage{bbm}%%%NOUVEAU PACK
\usepackage{mathabx}%%%NOUVEAU PACK
\newcommand{\BN}{{\mathbb{N}}}
\newcommand{\BR}{{\mathbb{R}}}

\newcommand{\BP}{{\mathbb{P}}}

\newcommand{\BE}{{\mathbb{E}}}

\newcommand{\gD}{\Delta}
\newcommand{\gd}{\delta}
\newcommand{\gb}{\beta}

\newcommand{\Ga}{\Gamma} %%
\newcommand{\gs}{\sigma} %%

\newcommand{\gO}{\Omega}
\newcommand{\go}{\omega}

\newcommand{\gl}{\lambda}

\newcommand{\gth}{\theta}
\newcommand{\gTh}{\Theta}

\newcommand{\gx}{\xi} %% nouvelle commande

\newcommand{\cD}{{\mathcal{D}}}

\newcommand{\cF}{{\mathcal{F}}}
\newcommand{\dd}{{\partial}}%%
\renewcommand{\hat}{\widehat}

\def\vp{\varphi}
\def\e{\varepsilon}

\newcommand{\ti}[1]{\tilde{#1}}

\newcommand{\ol}[1]{\overline{#1}}

\newcommand{\lips}{\text{Lip}}%%%%

\newcommand{\tprob}{\text{Prob}}%%%%

\newcommand{\half}{\frac{1}{2}}

\newtheorem{prop}{Proposition}[section]
\newtheorem*{prop*}{Proposition}
\newtheorem{thm}{Theorem}
\newtheorem*{thm*}{Theorem}
\newtheorem{lem}[prop]{Lemma}
\newtheorem{cor}[prop]{Corollary}

\theoremstyle{definition}

\newtheorem{defn}[prop]{Definition}
\newtheorem{rem}[prop]{Remark}
\newtheorem*{rem*}{Remark}

\title[Boundary representations]{Random walks in negative curvature and boundary representations.}
\author{Kevin Boucher}
\begin{document}

\maketitle
\begin{abstract}
%Let $\Ga$ be a non-elementary group of isometries of a pointed hyperbolic space $(X,d,o)$ and $\nu$ its Patterson-Sullivan measure on the Gromov boundary, $\dd X$, of $X$ at $o$.
%A boundary representation, $\pi:\Ga\rightarrow\mathcal{U}[L^2[\dd X,\nu]]$, is associated to the dynamical system $(\Ga,\dd X,\nu)$ that is known to be irreducible \cite{MR2787597},\cite{Garncarek:aa},\cite{MR2919980} when $\Ga$ acts geometrically on $(X,d)$.\\
In this paper we establish a version of the Margulis Roblin equidistribution theorem's \cite{MR2035655} \cite{MR2057305} for harmonic measures.
As a consequence a von Neumann type theorem is obtained for boundary actions and the irreducibility of the associated quasi-regular representations \cite{MR2787597} is deduced.
%This new phenomenon is used to deduce an original proof, based on a \text{flip operator} argument, of the irreducibility of boundary representations  associated to Gromov hyperbolic groups and fundamental groups of certain finite volume manifolds with pinched negative curvatures.
%As rank 1 recurrent subgroups discussed by Furstenberg in \cite{MR0284569} as lattices in $SO(n,1)$ acting on the sphere at infinity endowed with the Lebesgue measure.
%The approach proposed use another strategy for the irreducibility based on a \textit{flip operator} argument.
\end{abstract}

\section{Introduction}
Given a non-elementary discrete group, $\Ga$, acting properly by isometries on a proper roughly geodesic hyperbolic space $(X,d)$, two main classes of measures can be constructed  on the Gromov boundary $\dd X$ of $(X,d)$: the \textit{harmonic measures} and \textit{quasiconformal measures}.
Let us briefly recall these constructions.\\

A measure $\nu$ on $\dd X$ is called $\Ga$-quasiconformal if $\Ga$ preserves the class of $\nu$ and one can find $C\ge1$, for all $g\in\Ga$ and $\nu$-almost every $\xi\in\dd X$:
$$C^{-1}e^{\gd_\Ga\beta(o,g^{-1}.o;\xi)}\le\frac{dg_*\nu}{d\nu}(\xi)\le Ce^{\gd_\Ga\beta(o,g^{-1}.o;\xi)}$$
where $\gb$ denote the Buseman cocycle on $\dd X$.\\%for all $g\in\Ga$ and $\nu$-almost every $\xi\in\dd X$.\\
Given a point $o\in X$ in $(X,d)$, the critical exponent of $\Ga$ is defined as:
$$\gd_\Ga=\limsup_r\frac{1}{r}\ln|B_X(o,r)\cap \Ga.o|$$
It is strictly positive when $\Ga$ is non-elementary.

The Patterson-Sullivan procedure consists to take a weak limit in $\tprob(X\cup\dd X)$ when $s$ goes to $\gd_\Ga$ of the sequence of probabilities:
$$\frac{1}{H(s)}\sum_{g\in\Ga}e^{-sd(g.o,o)}\gd_{g.o}$$
for $s>\gd_\Ga$ and $H(s)=\sum_{g\in\Ga}e^{-sd(g.o,o)}$.
These Patterson-Sullivan measures are quasiconformal.
Moreover these are Hausdorff measures for a visual distance on $\dd X$ when $\Ga$ acts cocompactly on $(X,d)$ \cite{MR1214072}.\\

One the other hand given a probability measure, $\mu$, with 1st moment and support not contained in a elementary group.
If $(X_n)_n$ denote the random walk on $\Ga$ starting at the identity $e\in\Ga$, the sequence of $X$-valued random variables $(X_n.o)_n$ converges almost surely to a random variable $Z$ with values in $\dd X$.
The law of $Z$ is by definition the harmonic measure on $\dd X$ for the random walk generated by $\mu$.\\

The dynamical systems associated to these measures was extensively investigated these last years.
Initiated by U. Bader and R. Muchnik in \cite{MR2035655} a particular interest grown around the quasi-regular representations associated to these systems called \textit{boundary representations}.
These representations generalized the well known parabolic inductions $L^2(G/P)$ for semisimple Lie groups $G$ and $P$ minimal parabolic subgroups of $G$.
In their work the authors proved the irreducibility of boundary representations arising from Patterson-Sullivan measures and conjectured the following:\\

{\bf Conjecture [Bader-Muchnik]}
\textit{Let $G$ be a locally compact group and $\mu$ a admissible probability measure on $G$, the boundary representation of $G$ on the Poisson boundary $P(G,\mu)$ is irreducible.
}\\

Using the interplay between quasiconformal and harmonic measures in negative curvature \cite{MR2919980} \cite{MR2346273} we propose another approach of these questions that morally attempt to understand boundary dynamics in terms of spectral properties instead of geometric ones.
This led us to investigate probabilistic analogues of certain geometric phenomenons.

\subsection{Harmonic equidistribution and its consequences.}
Let $\mu$ be a symmetric probability measure on $\Ga$.
The Markov operator $P_\mu$ associated to a continuous action $\alpha:\Ga\times K\rightarrow K$ of $\Ga$ on a compact space $K$ is the positive contraction defined on the space of continuous functions $(C^0(K),\|.\|_\infty)$ given by:
$$P_\mu\vp(\xi)=\sum_{g\in\Ga}\mu(g)\vp(g\xi)$$
for $\vp\in C^0(K)$.
A probability measure, $\nu$, on $K$ is $\mu$-stationary if $\mu*\nu=\alpha_*(\mu\otimes\nu)=\nu$.
Moreover the measure $\mu$ is called $\nu$-ergodic if $\nu$ is stationary and every $P_\mu$-invariant function is $\nu$-almost everywhere constant.\\

The principal result of this paper is the following probabilistic version of the Margulis Roblin equidistribution theorem's \cite{MR2057305} \cite{MR2035655} for stationary measures on the Gromov boundary of $(X,d)$:\\

{\bf Theorem \ref{thm:equi}}
\textit{
Let $\nu$ be a probability measure on the Gromov boundary, $\dd X$, of $(X,d)$.
Assume $\mu\in\tprob(\Ga)$ is a 1st moment, non-elementary $\nu$-ergodic probability measure on $\Ga$.\\
Then the following equidistribution holds:
$$\BE^\mu[\vp(X_n.\xi)\psi(X_n^{-1}.\eta)]=\sum_{g\in\Ga}\vp(g\xi)\psi(g^{-1}\eta)\mu^{*n}(g)\xrightarrow{n\rightarrow+\infty} \int_{\dd X}\vp d\nu\int_{\dd X}\psi d\nu$$
for all $\vp,\psi\in C^0(\dd X)$ continuous functions and $\xi,\eta\in\dd X$.
}\\

Let $X$ be a Hadamard manifold with pinched curvature.
This means that $X$ is a complete simply connected Riemannian manifold of dimension greater than $2$ and curvature, $K_X$, which satisfies: $-b^2\le K_X\le -a^2$. %
The Brownian motion $(B_t)_t$ on $X$, that is the diffusion process generated by the Laplace-Beltrami operator on $X$, converges almost surely to a random variable  $B_\infty$ with values in the geometric boundary $\dd X$. %and its distribution gives a harmonic measure denotes $\nu$.
When the action $\Ga$ on $X$ has finite covolume W. Ballman and F. Ledrappier in \cite{MR1427756} refined the discretization procedure introduced by T. Lyons and D. Sullivan in \cite{MR755228} and earlier discussed by H. Furstenberg in \cite{MR0284569} which allows to interpret the distribution of $B_\infty$, also called harmonic, has a stationary measure for a random walk performed on $\Ga$:\\

{\bf Theorem [Ballman, Ledrappier \cite{MR1427756}]}
\textit{There exists a admissible symmetric probability measure on $\Ga$ with 1st moment such that the harmonic measure of the random walk generated by $\mu$ and the distribution of the limit random variable $B_\infty$ coincide.
}\\

This result guarantees the ergodicity of the measure $\mu$ and as a direct consequence of Theorem \ref{thm:equi} one has:\\

{\bf Corollary}
\textit{Let $X$ be a Hadamard manifold with pinched negative curvature and $\Ga$ a discrete group of isometries with finite covolume. 
Let $\nu$ be the unique harmonic measure on the visual boundary $\dd X$ of $X$ for the Brownian motion $(B_t)_t$.\\
Then $\Ga$ carries a admissible probability measure with 1st moment, $\mu$, such that:
$$\BE^\mu[\vp(X_n.\xi)\psi(X_n^{-1}.\eta)]=\sum_{g\in\Ga}\vp(g\xi)\psi(g^{-1}\eta)\mu^{*n}(g)\xrightarrow{n\rightarrow+\infty} \int_{\dd X}\vp d\nu\int_{\dd X}\psi d\nu$$
for all continuous functions $\vp,\psi\in C^0(\dd X)$ and $\xi,\eta\in\dd X$.
}\\

%When the $\Ga$-action on $(X,d)$ is cocompact and $\mu$ is admissible Kaïmanovich proved in \cite{MR1815698} the topological realization of the Poisson boundary P($\Ga,\mu$) on the Gromov boundary, $\dd X$, of $(X,d)$ with the unique $\mu$-stationary measure on $\dd X$ denoted $\nu$.
Back to a more general situation where $(X,d)$ is a proper roughly geodesic hyperbolic space.
Assuming the action of $\Ga$ on $(X,d)$ is cocompact and $\mu$ is finitely supported symmetric and admissible it was proved in \cite{MR2919980} that the harmonic measure, $\nu$, can be interpreted as a quasiconformal measure coming for the Green distance on $\Ga$ that satisfies hyperbolic properties.
Conversely when $(X,d)$ is a CAT($-1$) space \cite{MR1744486} every quasiconformal measure for $\Ga$ can be interpreted as a stationary measure for a symmetric probability measure with 1st moment on $\Ga$ \cite{MR2346273}.

This interplay motivate the investigation on quasiconformal measures that can be regarded as stationary measures.
%Moreover as explained below, certain fundamental groups of finite volume manifolds give examples of non cocompact groups where it is possible to understand the relation between certain quasiconformal and harmonic measures.\\
%%%%%%%%%%%%%%%%%%%%%%%%%%%%%%%%%%%%%%%%%%%%%%%%%%%%%%%%%%%%

A measure $\nu$ on a compact metric space $(K,d)$ is called Ahlfors-regular if there exist $\alpha>0$ and $C\ge1$ such that:
$$C^{-1}r^\alpha\le\nu(B(x,r))\le Cr^\alpha$$
for $0<r<\text{Diam}(K)$.\\% where $\text{Diam}(K)$ stands for the diameter of $K$.
Examples of such measures in our setting are given by quasiconformal measures on $\dd X$ associated to discrete groups which act geometrically on $(X,d)$.\\

Considering Ahlfors-regular quasiconformal measures on $\dd X$ for the action of $\Ga$ on $(X,d)$, Theorem \ref{thm:equi} together with techniques inspired by \cite{MR2346273} \cite{MR2787597} led us to the following von Neumann type result for measure class preserving actions:\\

{\bf Theorem \ref{thm:flipap}}%\label{thm:flipap}
\textit{
Let $\mu\in\tprob(\Ga)$ be a admissible finitely supported symmetric probability on $\Ga$ and $\nu$ its harmonic measure on $\dd X$.\\
Denote $\mathcal{P}_n$, with $n\ge0$, the operator defined as:
$$\mathcal{P}_n=\BE^\mu[\frac{\pi_o(X_n)}{\Xi_o(X_n)}]=\sum_{g\in\Ga}\frac{\pi_o(g)}{\Xi_o(g)}\mu^{*n}(g)$$
where $\Xi_o=(\pi_o{\bf 1}_{\dd X}|{\bf 1}_{\dd X})_{L^2}$ denotes the Harish-Chandra function of the boundary representation $\pi_o$ on $\Ga$.\\
Then for all $\vp,\psi\in L^2[\dd X,\nu]$, square integrable functions on $\dd X$, the following von Neumann type convergence holds:
$$(\mathcal{P}_n \vp|\psi)_{L^2}\xrightarrow[n\rightarrow\infty]{} \int_{\dd X}\vp d\nu\int_{\dd X}\psi d\nu$$
%&=\BE^\mu[(\frac{\pi_o(X_n)}{\Xi_o(X_n)}\otimes\frac{\pi_o(X_n^{-1})}{\Xi_o(X_n^{-1})} \vp_1\otimes\vp_2 |\psi_1\otimes\psi_2)]\\
}\\

In other words if $\mathcal{P}_{{\bf1}_{\dd X}}$ denote the one dimensional projector on the constant function ${\bf1}_{\dd X}$ given by:
$$\mathcal{P}_{{\bf1}_{\dd X}}(\vp)=(\vp|{\bf1}_{\dd X}){\bf1}_{\dd X}$$
for $\vp\in L^2[\dd X,\nu]$.
The Theorem \ref{thm:flipap} corresponds to the weak convergence of the sequence of operators $(\mathcal{P}_n)_n$ to $\mathcal{P}_{{\bf1}_{\dd X}}$.

%Note using spectral transfert techniques \cite{} it is known that $\|\pi_o(f)\|=\|\gl(f)\|$ for all $f\in\ell ^1(\Ga)$ in particular the assumption on the support of $\mu$ implies that $\|\pi_o(\mu^{*n})\|=\|\gl(\mu^{*n})\|\le\rho^n\xrightarrow{n\rightarrow+\infty}0$.

In opposite to the measure preserving situation, the vector ${\bf1}_{\dd X}$ appears to be cyclic for the boundary representation $(\Ga, \pi_o,L^2[\dd X,\nu])$ and a standard argument leads to the following:\\

{\bf Corollary \ref{cor:irr}}%\ref{}}
\textit{
%Let $\nu$ be a Ahlfors-regular quasiconformal measure for the $\Ga$-action.\\
Let $\mu\in\tprob(\Ga)$ be a admissible finitely supported symmetric probability on $\Ga$ and $\nu$ its harmonic measure on $\dd X$.
Then the associated boundary representation $(\Ga,\pi_o,L^2[\dd X,\nu])$ is irreducible.
}\\

%The spectral formulation of this criterion also allows to conclude in the situation of Hadamard manifolds with pinched negative curvatures where $\Ga$ has only finite covolume using the Lyons-Sullivan procedure:\\

%{\bf Corollary} 
%\textit{Let $X$ be a Hadamard manifold with pinched negative curvature and $\Ga$ a group acting properly on $X$ with finite covolume. 
%Assume the unique harmonic measure on the visual boundary, $\dd X$, for the Brownian motion $(B_t)_t$ is  Ahlfors-regular and $\Ga$-quasiconformal.
%Then the boundary representation $(\Ga,\pi_o,L^2[\dd X,\nu])$ is irreducible.
%}\\

%Examples can be obtained when $X$ has constant negative curvature where it is known that this harmonic measures in proportional to the Riemannian measure on the sphere at infinity.

\subsection{Outlines}
$ $\\
In Section \ref{sec:prelim}, we remind the reader standard facts about random walk on groups as well as some part of the theory of Patterson-Sullivan on quasiconformal measures.
Assuming no prior familiarity with this topic we propose a relatively self-contained account of the relevant ingredients needed for the understanding of the rest.
In Section \ref{sec:equi} we prove Theorem \ref{thm:equi}.\,% and discuss its consequences in terms of ergodicity of Markov operators.
The Section \ref{sec:main} is devoted to the proof of Theorem \ref{thm:flipap} which is splitted in two parts: the first part is a reduction from matrix coefficients analysis to equidistribution, cf. Proposition \ref{prop:coef}, and the second part is dedicated to the uniform boundedness of the sequence of operators $(\mathcal{P}_n)_n$ introduced in Theorem \ref{thm:flipap} statement.
%The last Section \ref{sec:} is devoted discrete Lyon-Sullivan procedure and the construction of ergodic measure supported on recurrent subgroups.

\subsection{Notations and terminologies}
$ $\\
Usually $(X,d)$ stands for a locally finite Gromov hyperbolic space pointed at $o\in X$ upon which a discrete group $\Ga$ acts geometrically.
An action of  $\Ga$ on $(X,d)$ is called \textit{geometric} if it is properly discontinuous and cocompact.

Except mention groups considered are assumed to be non elementary, i.e. non-virtually cyclic.

A probability measure $\mu\in\tprob(\Ga)$ on a group $\Ga$ is called admissible if its support generates $\Ga$ as a semi-group and symmetric if $\mu(g)=\mu(g^{-1})$ for all $g\in\Ga$.
Except mention in the rest every measures considered on $\Ga$ are assumed to have a 1st moment, i.e. the quantity $\sum_{g\in\Ga}d(g.o,o)\mu(g)$ is finite.\\

We will use the notations: $|x|_z=d(x,z)$ and $|g|_z=d(gz,z)$ for $x,z\in X$ and $g\in\Ga$.

In order to avoid the escalation of constants coming from estimates up to controlled additive or multiplication error terms we will use the following conventions.
Given two real valued functions, $a,b$, of a set $Z$, we write $a\preceq b$ if there exists $C>0$ such that $a(z)\le Cb(z)$ for all $z\in Z$ and $a\asymp b$ if $a\preceq b$ and $b\preceq a$.
Analogously we write $a\lesssim b$ if there exists $c$ such that $a(z)\le b(z)+c$ and $a\approx b$ if $a\lesssim b$ and $b\lesssim a$.\\

%%%%%%%%%%%%%%%%%%%%%%%%%%%%%%%%%%%%%%%%%%%%%%%%%%%%%%%%%%%%%%%%%%%
%%%%%%%%%%%%%%%%%%%%%%%%%%%%%%%%%%%%%%%%%%%%%%%%%%%%%%%%%%%%%%%%%%%
%%%%%%%%%%%%%%%%%%%%%%%%%%%%%%%%%%%%%%%%%%%%%%%%%%%%%%%%%%%%%%%%%%%
%%%%%%%%%%%%%%%%%%%%%%%%%%%%%%%%%%%%%%%%%%%%%%%%%%%%%%%%%%%%%%%%%%%

\section{Preliminaries}\label{sec:prelim}
\subsection{Hyperbolic spaces and compactifications.}
$ $\\
Given $(X,d)$ a discrete locally finite metric space, the \textit{Gromov product} on $X$ at the basepoint $o\in X$ is defined as:
$$(x|y)_o=\half(|x|_o+|y|_o-d(x,y))$$
for $x,y\in X$.
The space $(X,d)$ is called $(\gd)$-hyperbolic, with $\gd\ge0$, if it satisfies:
$$(x|y)_w\ge\min\{(x|z)_w,(z|y)_w\}-\gd$$
for all $x,y,z,w\in X$.\\

\begin{defn}
A discrete metric space $(X,d)$ is called roughly geodesic if there exists a constant $C$ such that for all pair of points $x,y\in X$, there exists a map $p:[0,T]\rightarrow X$ such that $p(0)=x$, $p(T)=y$ and
$$|t-t'|-C\le d(p(t),p(t'))\le|t-t'|+C$$
for all $t,t'\in[0,T]$.
\end{defn}
Note that hyperbolicity is preserved by quasi-isometries between proper roughly geodesic spaces \cite{Bonk2011}.\\

Given such a metric space, $(X,d)$, one can associate a compact space, $\dd X$, called Gromov boundary.
Let us recall briefly a construction of this object and its properties.

%%%%%%%%%%%%%%%%%%%%%%%%%%%%%%%%%%%%%%%%%%%%%%%%%%%%%%%%%%%%
%%%%%%%%%%%%%%%%%%%%%%%%%%%%%%%%%%%%%%%%%%%%%%%%%%%%%%%%%%%%
\subsubsection{Hyperbolic boundaries viewed as equivalence classes of sequences}
$ $\\
A sequence $(x_n)_n$ in $X$ \textit{goes to infinity} if $\lim_{n,m}(x_n|x_m)_o\rightarrow+\infty$.\\
The boundary $\dd X$ of $X$ can be defined as the equivalence classes on the set of sequences which go to infinity, $X^\infty\subset X^\BN$, endowed with the equivalence relation:
$$(a_n)\sim(b_n)\quad \text{if and only if}\quad \lim_{n,m}(a_n|b_m)_o\rightarrow\infty$$
with $(a_n)_n,(b_n)_n\in X^\infty$.\\
This construction does not depend on the choice of the basepoint $o$.
The class of the sequence $(x_n)_n$ is denoted $\lim_nx_n=\xi$.\\

Identifying stationary sequences with elements of $X$ the Gromov product extends to $\ol{X}=X\cup\dd X$ by the formula:
$$(x|y)_o=\sup_{(x_n),\,(y_m)}\liminf_{n,m}(x_n|y_m)_o$$
where the $\sup$ is taken over all sequences $(x_n)_n$ and $(y_m)_m$ that represent respectively $x$ and $y$ in $\ol{X}$.\\

This extended product satisfies the following properties \cite{MR1744486}:
\begin{enumerate}
\item $(x|y)_o=\infty$ if and only if $x,y\in\dd X$ and $x=y$;
\item $(x|y)_o\ge\min\{(x|z)_o,(z|y)_o\}-2\gd$ for all $x,y,z\in \ol{X}$;
\item for all $\xi,\eta\in\dd X$ and $(x_n)_n,\, (y_m)_m$ with $\lim_nx_n=\xi$ and $\lim_my_m=\eta$ one has:
$$(\xi|\eta)_o-2\gd\le\liminf_{n,m}(x_n|y_m)_o\le (\xi|\eta)_o.$$
\end{enumerate}

Let $\e>0$ and define the kernel:
$$\rho_{\e,o}:\ol{X}\times \ol{X}\to\BR_+,\quad \rho_\e(x,y)=
\begin{cases}
e^{-\e(x|y)_o}&\text{$(x,y)\in \ol{X}^{\times 2}\setminus \gD_X$}\\
0&\text{otherwise}
\end{cases}$$
where $\gD_X=\{(x,x)\,|\,x\in X\}$ is the diagonal points that belong to $X$.\\

Using the Gromov product properties observe that:
\begin{enumerate}
\item $\rho_{\e,o}(x,y)=\rho_{\e,o}(y,x)$ for all $x,\,y\in \ol{X}$;
\item $\rho_{\e,o}(x,y)=0$ if and only if $x=y$;
\item for all $x,y,z\in\ol{X}$:
$$\rho_{\e,o}(x,y)\le(1+\e')\max\{\rho_{\e,o}(x,y),\rho_{\e,o}(y,z)\}$$
with $\e'=e^{2\gd\e}-1$.
\end{enumerate}

The \textit{chains-pseudo-distance} associated is defined as:
$$\gth_{\e,o}(x,y)=\inf\{\sum_{i=1}^n\rho_{\e,o}(x_i,x_{i+1})\mid x_1=x,\dots,x_{n+1}=y\}$$
and for $\e'<\sqrt{2}-1$ satisfies:
$$(1-2\e')\rho_{\e,o}(x,y)\le\gth_{\e,o}(x,y)\le\rho_{\e,o}(x,y)$$
for all $(x,y)\in \ol{X}^{\times 2}\setminus \gD_X$.\\
In particular for such choice of $\e$, $\gth_{\e,o}$ is a distance on $\ol{X}$.
In the rest $\e'$ is assumed to be strictly lower than $\sqrt{2}-1$.\\
Since $e^{\e|x|_o}\le\rho_{\e,o}(x,y)$ for all $x\neq y$ in $X$ and the induced topology on $X$ by $\gth_{\e,o}$ is discrete.
\begin{defn}
Let $(X,d)$ be a proper roughly geodesic hyperbolic space.
A distance $d'$ on boundary $\dd X$ is called visual if there exist $\e>0$, $\gl\ge1$ and $o\in X$ such that:
$$\gl^{-1}\rho_{\e,o}(\xi,\eta)\le d'(\xi,\eta)\le \gl\rho_{\e,o}(\xi,\eta)$$
for all $\xi,\eta\in \dd X$.
\end{defn}
\noindent Note that the restriction of $\gth_{\e,o}$ to $\dd X$ is a visual metric. Its restriction is denoted $d_{o,\e}$ or simply $d_{o}$ in the rest .\\

A proper sequence $(x_n)_n$ in $X$ is a Cauchy sequence for $\gth_{\e,o}$ if and only if 
$$\lim_{n,m}(x_n|x_m)_o=+\infty$$
and another proper sequence converges to the same boundary point for $\gth_{\e,o}$ if and only if
$$\lim_{n,m}(x_n|y_m)_o=+\infty$$

In other words the metric space $(\ol{X},\gth_{\e,o})$ can be seen as the completion of $(X,\gth_{\e,o}|_X)$.\\

Endowed with this topology $\ol{X}=X\cup\dd X$ satisfies the following properties:
\begin{enumerate}
\item The topological space $\ol{X}$ is compact;
\item The embedding $X\hookrightarrow\ol{X}$ is a homeomorphism on his image, in particular $X$ is open in $\ol{X}$;
\item for all sequence $(\xi_n)_n$ in $\dd X$ one has $\xi_n\rightarrow \xi\in\dd X$ if and only if $(\xi_n|\xi)_o\rightarrow\infty$.
\end{enumerate}
\subsection{Busemann compactification and strongly hyperbolic metric spaces}%%%
\begin{defn}
A proper roughly geodesic metric space $(X,d)$ is strongly hyperbolic if there exists $\e>0$ such that
$$\exp(-\e(x|y)_w)\le \exp(-\e(x|z)_w)+\exp(-\e(z|y)_w)$$
for all $x,y,z,w\in X$.
\end{defn}
Some remarkable consequences of the strong hyperbolicity \cite{MR3551185} are:
\begin{enumerate}
\item the Gromov product $(.|.)_w$ extends continuously to $\ol{X}=X\cup\dd X$ for all $w\in X$;
\item the kernel $\rho_{\e,o}$, introduced above, is an actual metric on $\dd X$ for $\e$ small enough.
\end{enumerate}
These facts guarantee sharper conformal properties on the boundary $\dd X$.\\

Given a strongly hyperbolic space $(X,d)$, the Busemann function at $\xi\in\dd X$ is well defined as:
$$\gb_\xi(x,y)=2(\xi|y)_x-|y|_x=\lim_{z\rightarrow\xi}d(z,x)-d(z,y)$$
for all $x,y\in X$. In particular the Busemann boundary $\dd_\infty X$ and the Gromov boundary $\dd X$ coincide \cite{BusemannHerbert1955Tgog}.\\
In the rest $\gb_\xi(o,x)$ is denoted $\gb_\xi(x)$ for all $\xi\in\dd X$ and $x\in X$.\\

Another important feature of the strong hyperbolicity is the \textit{metric conformality} for the action of the group of isometries of $X$ on $\dd X$:
$$d_{o,\e}(g\xi,g\eta)=\exp(\frac{\e}{2}\gb_\xi(g^{-1}o))\exp(\frac{\e}{2}\gb_\eta(g^{-1}o))d_{o,\e}(\xi,\eta)$$
for all $\xi,\eta\in\dd X$ and $g\in\text{Isom}(X,d)$, where $d_{o,\e}=\rho_{\e,o}$.\\

As a consequence if $\nu$ denote the Hausdorff measure of dimension $D_\e$ on $(\dd X,d_{o,\e})$ \cite{MR3551185}, the following \textit{measure conformal} relation holds:
$$\frac{dg_*\nu}{d\nu}(\xi)=\exp(\e D_\e \gb_\xi(g^{-1}o))=\exp(\gd_\Ga \gb_\xi(g^{-1}o))$$
for all $\xi\in\dd X$ and $g\in\text{Isom}(X,d)$.\\
\begin{rem}
The measure $\nu$ is independent of the parameter $\e$ chosen.\\
Examples of strongly hyperbolic space are given by Hadamard manifolds with pinched negative curvatures more generally CAT($-1$) metric spaces \cite{MR3551185} and, as exposed in subsection \ref{sub:green}, Green metric structures on hyperbolic groups.
\end{rem}

%%%%%%%%%%%%%%%%%%%%%%%%%%%%%%%%%%%%%%%%%%%%%%%%%%%%%%%%%%%%
%%%%%%%%%%%%%%%%%%%%%%%%%%%%%%%%%%%%%%%%%%%%%%%%%%%%%%%%%%%%
\subsection{Boundary retractions}
$ $\\
As above $(X,d)$ is a proper roughly geodesic hyperbolic space.\\
A \textit{boundary retraction} is defined as a continuous map $f:\ol{X}\to\dd X$ such that $f|_{\dd X}=id_{\dd X}$.
Such a retraction induces an isometric operator:
$$E_f:(C^0(\dd X),\|.\|_\infty)\rightarrow (C^0(\ol{X}),\|.\|_\infty),\quad \vp\mapsto\vp\circ f$$
,i.e. $\|\vp\circ f\|_\infty=\|\vp\|_\infty$ for all $\vp\in C^0(\dd X)$, called \textit{boundary extension}.
%Note that $E_f(\vp)|_{\dd X}=\vp$ for all $\vp\in C(\dd X)$.

%%%%%%%%%%%%%%%%%%%%%%%%%%%%%%%%%%%%%%%%%%%%%%%%%%%%%%%%%%%%
%%%%%%%%%%%%%%%%%%%%%%%%%%%%%%%%%%%%%%%%%%%%%%%%%%%%%%%%%%%%
\subsubsection{Existence of boundary retractions}
$ $\\
The \textit{shadow} of a ball centered at $x\in X$ with radius $R$ viewed at the basepoint $o\in X$ is defined as:
$$O_o(x,R)=\{\xi\in\dd X|\text{$\exists c$ s.t $c(0)=o$, $c(\infty)=\xi$, $c(\BR_+)\cap B(x,R)\neq\emptyset$}\}$$
where $c$ denote a rough geodesic in $X$.\\
A shadow can equivalently be described in terms of the Gromov product as: 
$$O_o(x,R)=\{\xi\in\dd X|\,\text{$(\xi|x)_o\ge |x|_o-R$}\}$$
Note that there exists $R_0>0$ which does not depend on $x\in X$ such that for all $R\ge R_0$, $O_o(x,R)\neq\emptyset$ \cite{MR1214072}.\\

An example of a boundary retraction can be obtained as follows:
\begin{lem}
Let $R>0$ be a large positive constant and $f:X\rightarrow \dd X$ such that $f(x)\in O_o(x,R)$ for all $x\in X$.
Then $f$ is well defined and extends to a boundary retraction on $\ol{X}$ called linear boundary retraction.
\end{lem}
\begin{proof}
Since $X$ is discrete it is enough to prove for every sequence $(x_n)_n$ in $X$ converging to a point $\xi\in \dd X$, $(f(x_n))_n$ converges to $\xi$.\\
The definition of $f$ implies $(f(x)|x)_o\ge |x|_o-R$ for all $x\in X$.
On the other hand $\gth_{\e,o}(y_n,\eta)\rightarrow 0$ if and only if $(y_n|\eta)_o\rightarrow +\infty$ for any $(y_n)_n$ in $\ol{X}$ and $\eta\in\dd X$.\\
Thus given $(x_n)_n$ with $x_n\rightarrow \xi$ since:
$$(f(x_n)|\xi)_o\ge\min\{(f(x_n)|x_n)_o,(x_n|\xi)_o\}-\gd$$
one has $\gth_{\e,o}(f(x_n),\xi)\rightarrow 0$, in other words $f$ can be extended to $\ol{X}$ by $f(\xi)=\xi$ for all $\xi\in\dd X$.
\end{proof}

In some sense every boundary retraction is of this form:
\begin{lem}
Let $f$ be a boundary retraction of $X$.
Then there exists a proper function $\vp:\BR_+\rightarrow \BR_+$ such that:
$$(f(x)|x)_o\ge\vp(|x|_o)$$
for all $x\in X$.
\end{lem}
\begin{proof}
Let $R>0$ be a positive constant and $f'$ a linear boundary retraction such that $(f'(x)|x)_o\ge|x|_o-R$ for $x\in X$.\\
Observe that:
$$(f(x)|x)_o\ge\min\{(f'(x)|x)_o,(f(x)|f'(x))_o\}-\gd\ge\vp(|x|_o)$$
with
$$\vp(t)=\inf_{\lfloor t\rfloor \le|x|<\lfloor t\rfloor +1}[\min\{\lfloor t\rfloor-R,(f(x)|f'(x))_o\}-\gd]$$
Therefore $\vp$ is proper iff the map $x\mapsto (f(x)|f'(x))_o$ on $X$ is proper.\\
A similar argument than the one used in the proof of the next lemma, based on compactness of $\dd X$, the continuity of $f$ and $f'$ together with the fact that these two functions are identical on the boundary guarantees the properness of the last map.
\end{proof}

%%%%%%%%%%%%%%%%%%%%%%%%%%%%%%%%%%%%%%%%%%%%%%%%%%%%%%%%%%%%
%%%%%%%%%%%%%%%%%%%%%%%%%%%%%%%%%%%%%%%%%%%%%%%%%%%%%%%%%%%%
\subsubsection{Boundary retraction and sequence of measures with accumulation at infinity}
$ $\\
A sequence of probabilities on $\ol{X}$ has \textit{accumulation at infinity} if for all compact $K\subset X$ one has $\mu_n(K)\rightarrow 0$.
In other words any limit of a $(\mu_n)_n$ subsequence is supported in $\dd X$.
\begin{lem}
Let $f$ be a boundary retractions and $(\mu_n)_n$ a sequence of probabilities on $\ol{X}$ with accumulation at infinity.\\
If $f_*\mu_n$ converges weakly to $\mu\in\tprob(\dd X)$, then for every other boundary retraction $f'$, $f'_*\mu_n$ converges weakly to $\mu$.
\end{lem}
\begin{proof}
Let $f'$ be another boundary retraction.
Since $E_{f}(\vp)$ and $E_{f'}(\vp)$ are continuous functions on the compact space $\ol{X}$ and thus uniformly continuous.
Moreover $E_{f}(\vp)|_{\dd X}=E_{f'}(\vp)|_{\dd X}=\vp$.\\
Therefore given $\e>0$, there exists $\alpha>0$ such that for all $\xi\in\dd X$ and $x\in B_{\gth_{\e,o}}(\xi,\alpha)\subset\ol{X}$, $|E_{f}(\vp)(x)-\vp(\xi)|,\,|E_{f^{\prime}}(\vp)(x)-\vp(\xi)|\le\frac{\e}{4}$.\\
Because $\dd X$ is compact one can find a finite number, $n$, of $\xi_i\in\dd X$, $i=1,\dots, n$ such that $\dd X\subset \cup_{i=1}^nB_{\gth_{\e,o}}(\xi_i,\alpha)=U$.\\
By construction of $U\subset\ol{X}$, $|E_{f^{\prime}}(\vp)(x)-E_{f}(\vp)(x)|\le\half\e$ for all $x\in U$.\\
Let $K=U^c\subset X$ be the compact complement of $U$ in $\ol{X}$.
Since $\mu_n$ has accumulation at infinity, there exists $n_0$, for all $n\ge n_0$, $\mu_n(K)\le\frac{\e}{4\|\vp\|_\infty}$.
Together these estimates lead to:
\begin{align*}
|\int_{\dd X}\vp\, f_*\mu_n-\int_{\dd X}\vp\, f'_*\mu_n|&\le |\int_KE_{f}(\vp)-E_{f^{\prime}}(\vp)\,d\mu_n|+|\int_UE_{f}(\vp)-E_{f^{\prime}}(\vp)\,d\mu_n|\\
&\le \mu_n(K)(\|E_{f^{\prime}}(\vp)\|_\infty+\|E_{f}(\vp)\|_\infty)+\half\e\mu_n(U)\\
&\le 2\mu_n(K)\|\vp\|_\infty+\half\e\le \e
\end{align*}
which concludes the proof.
\end{proof}

In the rest a boundary retraction, $f$, is fixed and for all $x\in X$ we denote $\hat{x}=f(x)\in\dd X$.
Moreover we extend this notation for $m\in\tprob(\ol{X})$ (respectively, $\vp\in C^0(\dd X)$) by $\hat{m}=f_*m\in\tprob(\dd X)$ (respectively, $\hat{\vp}=E_f(\vp)\in C^0(\ol{X})$).\\
Note that for all $\vp\in C^0(\dd X)$, $x\in \ol{X}$ and $m\in\tprob(\ol{X})$ the following identities hold: 
$$\hat{\vp}(x)=\vp(\hat{x}),\quad \hat{m}(\vp)=m(\hat{\vp}).$$

%%%%%%%%%%%%%%%%%%%%%%%%%%%%%%%%%%%%%%%%%%%%%%%%%%%%%%%%%%%%
%%%%%%%%%%%%%%%%%%%%%%%%%%%%%%%%%%%%%%%%%%%%%%%%%%%%%%%%%%%%INTRODUCE THE MATERIAL NEEDED IN THE REST.
\subsection{Random walks and hyperbolic groups.}\label{sub:green}
$ $\\
In the first part of this subsection $\Ga$ is a general discrete group .\\

Given a probability $\mu\in\tprob(\Ga)$ on $\Ga$ a (left) $\mu$-random walk on $\Ga$ is a Markov chain with states space $\Ga$ and transition probabilities $p_\mu(g,h)=\mu(g^{-1}h)$.\\
We call \textit{$\mu$-random walk starting at $g\in\Ga$} the unique Markov chain (up to equivalence) with initial law $\gd_g$ and transition probabilities $p_\mu$.\\

\subsubsection{Standard models of random walks}
$ $\\
Let $(\Ga^{\BN^{*}},\mathcal{B},\mu^{\BN^{*}})$ be the probability space obtained as infinite product of $(\Ga,\mathcal{P}(\Ga),\mu)$, where $\mathcal{B}$ is the $\sigma$-algebra generated by cylinders.
The coordinate projectors 
$$H_k:\Ga^{\BN^{*}} \rightarrow\Ga,\quad H_k((\go_i)_i)=\go_k$$
form a sequence of $\mu$ distributed independent random variables and a model for random walk starting at $g_0$ is given by
$X_0=g_0$ and $X_n=H_1\dots H_n$ for $n\ge 1$.
In other words:
$$\BP(X_{n+1}=h|X_0=h_0,\dots,X_n=g)=
\begin{cases}
0&\text{if $h_0\neq g_0$}\\
p_\mu(g,h)&\text{if $h_0=g_0$}
\end{cases}$$
The canonical model for the random walk starting at $g\in\Ga$ is defined as the image of $(\Ga^{\BN^{*}},\mu^{\BN^{*}})$ by:
$$\mathcal{R}_g:\Ga^{\BN{*}}\rightarrow \Ga^{\BN},\quad \mathcal{R}_g((\go_i)_i)=
\begin{cases}
g&\text{for $n=0$}\\
gH_1\dots H_n((\go_i)_i)&\text{$n\ge1$}
\end{cases}
$$
where the random position at time $n$, $X_n$, is given by the $n$-th coordinate projection on $\Ga$.\\
In the rest we denote $\gO=\Ga^{\BN}$ endowed with $\gs$-algebra generated by cylinders and the probability:  
$$\BP={\mathcal{R}_e}_*\mu^{\BN_+}=\gd_e\otimes\bigotimes_{n=1}^\infty\mu^{*n}$$
or $\BP^\mu$ to insist on the $\mu$ dependence.
\begin{rem}
If $(X_n)_n$ is a $\mu$-random walk starting at $e\in\Ga$, then $(gX_n)_n$ gives a $\mu$-random walk starting at $g\in\Ga$. 
Moreover the corresponding law on $\gO$ is given by $\BP_g=g_*\BP={\mathcal{R}_g}_*\mu^{\BN_+}=\gd_g\otimes\bigotimes_{n=1}^\infty g_*\mu^{*n}$ where $\Ga$ acts on $\gO$ by the left-diagonal translation.
\end{rem}

%%%%%%%%%%%%%%%%%%%%%%%%%%%%%%%%%%%%%%%%%%%%%%%%%%%%%%%%%%%%
%%%%%%%%%%%%%%%%%%%%%%%%%%%%%%%%%%%%%%%%%%%%%%%%%%%%%%%%%%%%
\subsubsection{Filtrations and stopping times}\label{subsec:filt}
$ $\\
A filtration on a measurable space $(\gO,\mathcal{B})$ is an increasing sequence of sub-$\gs$-algebra of $\mathcal{B}$, $(\cF_n)_n$.\\
In our case the \textit{canonical filtration} is defined by $\cF_n=\gs(X_0,\dots,X_n)$ for $n\ge0$ where $(X_n)_n$ is a $\mu$-random walk.\\
Given $(\gO,\mathcal{B},(\cF_n)_n)$ a measurable map $T:\gO\rightarrow \BN\cup\{\infty\}$ is called \textit{stopping time} with respect to the filtration $(\cF_n)_n$ if $\{T=n\}\in\mathcal{F}_n$ for all $n$.

%%%%%%%%%%%%%%%%%%%%%%%%%%%%%%%%%%%%%%%%%%%%%%%%%%%
Together with this stopping time $T$ is associated the \textit{stopping time $\gs$-algebra}, $\cF_T$, defined as:
$$\cF_T=\{B\in\mathcal{B}\,|\,\text{$B\cap\{T=n\}\in\cF_n$ for all $n\ge0$}\}.$$
Assuming $T$ is almost surely finite the \textit{$\mu$-random position at time $T$} is the $\cF_T$-measurable random variable given by:
$$X_T: \gO\rightarrow\Ga,\quad X_T(\go)=
\begin{cases}
X_n(\go)&\text{if $T(\go)=n$}\\
e&\text{if $T(\go)=+\infty$}
\end{cases}$$
%%%%%%%%%%%%%%%%%%%%%%%%%%%%%%%%%%%%%%%%%%%%%%%%%%%%%%%%%%%%

%%%%%%%%%%%%%%%%%%%%%%%%%%%%%%%%%%%%%%%%%%%%%%%%%%%%%%%%%%%%
%%%%%%%%%%%%%%%%%%%%%%%%%%%%%%%%%%%%%%%%%%%%%%%%%%%%%%%%%%%%
\subsubsection{Fundamental transformations associated to random walks}\label{subsub:transf}
$ $\\
Given $g\in\Ga$ the map $\mathcal{R}_g$ defined above is an isomorphism of measured spaces with inverse $\mathcal{R}_g^{-1}((\go_n)_n)=(X_{i-1}^{-1}X_i((\go_n)_n))_{i\ge1}$.\\
The shift, $T$, on $(\Ga^{\BN^{*}},\mu^{\BN^{*}})$ given by $T((\go_i)_i)=(\go_{i+1})_i$ is a measure preserving transformation that is mixing and thus ergodic.
Moreover $T$ is intertwined by $\mathcal{R}_g$ with the Bernoulli shift $U_g$ on $(\gO,\BP_g)$ given by $U_g((\go_n)_n)=(g\go^{-1}\go_{n+1})_n$.
Since $\mathcal{R}_g$ is an $(T,U_g)$-equivariant isomorphism of measured spaces it preserves spectral properties and thus $U_g$ is also ergodic and mixing on $(\gO,\BP_g)$.\\

Another transformation associated to a Markov chains is the \textit{Markov shift}:
$$\gTh:(\gO,\BP_g)\rightarrow (\gO,\BP_g),\quad \gTh((\go_n)_n)\rightarrow (\go_{n+1})_n.$$
More generally, given a stopping time $T$ one can define on $\{T<+\infty\}\subset\gO$  the Markov shift at time $T$ by:
$$\gTh_T:\{T<+\infty\}\subset\gO\rightarrow \gO,\quad, \gTh_T(\go)=(\go_{T(\go)+n})_n$$
Note that unlike the Bernoulli shift these Markov transformations are not measure preserving.\\
Nevertheless the so called strong Markov property holds:
\begin{prop*}[Strong Markov]
Let $(\gO,\mathcal{B},(\cF_n)_n,\BP)$ be filtered space and $\vp$ a positive measurable function.
Then the following relation between random variables holds:
$$\BE(\mathbbm{1}_{\{T<+\infty\}}\vp\circ\gTh_T|\cF_T)=\mathbbm{1}_{\{T<+\infty\}}\BE_{X_T}(\vp)$$
\end{prop*}
%Note that this is a relation between random variables.
%\begin{prop*}
%Let $(\gO,\BP_g)$ be as above and $f$ a positive measurable function on $\gO$.
%Then:
%$$\BE_g(f\circ\gTh^n|X_0,\dots,X_n)=\BE_{X_n}(f)$$
%where $\BE_g$ is the expectation associated to the probability $\BP_g$.
%\end{prop*}

%%%%%%%%%%%%%%%%%%%%%%%%%%%%%%%%%%%%%%%%%%%%%%%%%%%%%%%%%%%%
%%%%%%%%%%%%%%%%%%%%%%%%%%%%%%%%%%%%%%%%%%%%%%%%%%%%%%%%%%%%
\subsubsection{Random walks and metric structures on hyperbolic groups}
$ $\\
%We restrict our discussion on the class of hyperbolic groups.\\
A discrete group $\Ga$ is called Gromov hyperbolic if it acts geometrically on a proper roughly geodesic $\gd$-hyperbolic space $(X,d)$.\\
Let $\cD(\Ga)$ be the collection of left-invariant pseudo-metrics on $\Ga$ which are quasi-isometric to a locally finite word distance.\\
Since $\Ga$ acts geometrically on $(X,d)$, the Milnor lemma guarantees that for any locally finite word distance, $d_\Ga$, on $\Ga$ and $w\in X$ the map:
$$\vp:(\Ga,d_\Ga)\rightarrow (O(w),d_X|_{O(w)})\subset (X,d_X),\quad g\mapsto g.w$$
is a $\Ga$-equivariant quasi-isometry. 
In particular the distance $d_X|_{O(w)}$ on $O(w)\simeq_{\text{q.i}}\Ga$ is hyperbolic and belongs to $\cD(\Ga)$.
Therefore we will assume in the rest that $X=\Ga$ endowed with a distance $d\in\mathcal{D}(\Ga)$ upon which $\Ga$ acts by left-translations.\\
In order to make a distinction between $\Ga$ as a group acting on a space and $\Ga$ as metric space we denote the last one $(X,d)$ and fixe the basepoint $o=e\in X$.\\

Let us introduce an important example of strong hyperbolic distances on $\Ga=X$ which arise from random walks called \textit{Green distance}.\\
%Assume $\Ga$ is non-elementary, i.e. non virtually cyclic.
Let $\mu\in\tprob(\Ga)$ be a admissible symmetric probability measure on $\Ga$ and let us consider the stopping time:
$$\tau_g:\gO\rightarrow \BN\cup\{\infty\},\quad \tau_g((\go)_n)=\inf\{n|\go_n=g\}.$$
The probability that a $\mu$-random walk starting at $g$ ever hits $g'$ is given by:
$$F_\mu(g,g')=\BP_g(\tau_{g'}<+\infty)$$
and the associated Green distance is defined as $d_\mu(g,g')=-\log F_\mu(g,g')$.\\ 
It is enough for $\Ga$ to be finitely generated non-amenable to guarantee that $d_\mu$ is a left-invariant distance quasi-isometric to locally finite word distances \cite{MR2278460}. 

Nevertheless the rough geodesic structure which is needed to deduce the hyperbolicity uses the Ancona criterion \cite{MR1100282} which holds when $\mu$ is finitely supported and $\Ga$ is hyperbolic:
\begin{thm*}[Ancona \cite{MR1100282}]
Given a non-elementary hyperbolic group $\Ga$, the Green distance induced by a admissible finitely supported symmetric probability on $\Ga$ is hyperbolic.
\end{thm*}
The strong hyperbolic property of such metrics was established in \cite{MR3551185}:
\begin{thm*}[Bogdan, Spakula \cite{MR3551185}]
Given a non-elementary hyperbolic group $\Ga$, the Green distance induced by a admissible finitely supported symmetric probability on $\Ga$ is strongly hyperbolic.
\end{thm*}

%%%%%%%%%%%%%%%%% AAAA MODIFIERRRRRR%%%%%%%%%%%%%%%%%%%%%%%%%%%%%
The strong hyperbolicity of these type of distances together with the Patterson-Sullivan theory guarantee the conformal properties of stationary measures associated Theorem 1.5 \cite{MR2919980}:
$$\frac{dg_*\nu}{d\nu}(\xi)=\exp(\gd_\Ga\gb_\xi(g^{-1}o))$$
for $\xi\in\dd X$ and $g\in\text{Isom}(X,d)$, where $\gd_\Ga$ is the critical exponent of $\Ga$.

%%%%%%%%%%%%%%%%%%%%%%%%%%%%%%%%%%%%%%%%%%%%%%%%%%%%%%%%%%%%%
%%%%%%%%%%%%%%%%%%%%%%%%%%%%%%%%%%%%%%%%%%%%%%%%%%%%%%%%%%%%
%%%%%%%%%%%%%%%%% AAAA MODIFIERRRRRR%%%%%%%%%%%%%%%%%%%%%%%%%%%%%
%%%%%%%%%%%%%%%%%%%%%%%%%%%%%%%%%%%%%%%%%%%%%%%%%%%%%%%%%%%%
%%%%%%%%%%%%%%%%%%%%%%%%%%%%%%%%%%%%%%%%%%%%%%%%%%%%%%%%%%%%
%%%%%%%%%%%%%%%%%%%%%%%%%%%%%%%%%%%%%%%%%%%%%%%%%%%%%%%%%%%%
%%%%%%%%%%%%%%%%%%%%%%%%%%%%%%%%%%%%%%%%%%%%%%%%%%%%%%%%%%%%
\subsubsection{Spectral gaps and random walks}\label{subsub:drift}
$ $\\
A symmetric probability measure on $\Ga$ is called non-elementary if the semigroup generated by its support in $\Ga$ is non-elementary.

Given an symmetric non-elementary probability measure, $\mu$, on $\Ga$, the spectral radius, $\rho(\mu)$, is defined as:
$$\rho(\mu)=\limsup_n\mu^{*n}(e)^{\frac{1}{n}}.$$
It is related to the amenability of the subgroup generated its support by the following theorem:
\begin{thm*}[Kersten \cite{MR109367}]
Let $\Ga$ be a finitely generated group and $\mu$ a symmetric probability on $\Ga$.\\
Then $\mu$ has a spectral gap, i.e. $\rho(\mu)<1$ if and only if $\mu$ is non-elementary.
\end{thm*}

%\begin{rem}
%In the preceding theorem the probability measure $\mu$ does not need to be finitely supported and the subgroup $H$ generated by its support might be infinitely generated.
%\end{rem}

Given a $\mu$-random walk starting at $e\in\Ga$, $(X_n)_n$, for all $n,m\ge0$ the following subadditive inequality holds:
$$|X_{n+m}|_o\le|X_n|_o+|X_m\circ U^n|_o$$
Using the ergodicity of the Bernoulli shift, $U$, the Kingman subadditive \cite{MR704553} theorem guarantees the convergence almost surely and in $L^1(\gO,\BP)$ of $(\frac{|X_n|_o}{n})_n$ to a constant $\ell$ called the \textit{drift of escape}, with $\ell>0$ whenever $\mu$ is non-elementary \cite{MR1743100}.\\

Let us prove the following useful estimate:
\begin{lem}
Let $\mu\in\tprob(\Ga)$ be a probability measure on a group $\Ga$.
Then for all $n$ one has:
$$\mu^{*n}(g)\le \max\{\mu^{*{n-1}}(e),\mu^{*n}(e)\}$$
for all $g\in\Ga$.
\end{lem}
\begin{proof}
Observe that:
\begin{align*}
\mu^{*2n}(x)&=\sum_{g\in\Ga}\mu^{*n}(g^{-1}g')\mu^{*n}(g')\\
&\le\sqrt{\sum_{g\in\Ga}\mu^{*n}(g^{-1}g')^2}\sqrt{\sum_{g\in\Ga}\mu^{*n}(g')^2}\\
&\le\sum_{g\in\Ga}\mu^{*n}(g')^2=\mu^{*2n}(e)
\end{align*}
and 
\begin{align*}
\mu^{*2n+1}(x)&=\sum_{g\in\Ga}\mu^{*2n}(g^{-1}g')\mu(g')\\
&\le\sum_{g\in\Ga}\mu^{*2n}(e)\mu(g')=\mu^{2n}(e)\\
\end{align*}
which imply the lemma.
\end{proof}

%Remark that if $\Ga$ is assumed to be non-amenable the Tchebychev inequality gives:
%$$\BP(|\frac{|X_n|_o}{n}-\ell|\ge\frac{\ell}{2})\le\frac{2}{\ell}\|\frac{|X_n|_o}{n}-\ell\|_{L^1(\gO,\BP_e)}$$
%and thus 
%$$\sum_{g\in B(o,\half\ell n)}\mu^{*n}(g)=\BP(|X_n|_o\le\half\ell n)\rightarrow 0.$$
%This fact is used latter in the proof of Thereom \ref{}.
%In other words the sequence of probabilities $(\mu^{*n})_n$ has accumulation at infinity.

%%%%%%%%%%%%%%%%%%%%%%%%%%%%%%%%%%%%%%%%%%%%%%%%%%%%%%%%%%%%
%%%%%%%%%%%%%%%%%%%%%%%%%%%%%%%%%%%%%%%%%%%%%%%%%%%%%%%%%%%%
\subsubsection{Random walks on Hyperbolic groups}
$ $\\
%A hyperbolic group $\Ga$ acting geometrically on $(X,d)$ is called \textit{non-elementary} if $|\dd X|>2$, in particular it is non-amenable.\\
In addition to recalling basics about the asymptotic behavior of random walks over hyperbolic spaces we insist on ideas present in Theorem \ref{thm:equi} proof.
Assume $\Ga$ is a non-elementary group that acts properly by isometries on $(X,d)$ and $\mu\in\tprob(\Ga)$ an symmetric non-elementary probability on $\Ga$.
%It follows that $\mu$ has a 1-th moment, i.e.:
%$$\sum_{g\in\Ga}\mu(g)|g|_o<+\infty.$$
%Moreover $\mu$ is assumed to be symmetric, i.e. $\mu(g)=\mu(g^{-1})$ for all $g\in \Ga$.

\begin{prop}\label{lem:asconv}
For $\BP$-almost every $\go\in\gO$ the sequence $(X_n(\go).o)_n$ in $(X,d)$ converges to $Z(\go)\in\dd X$.
\end{prop}
\begin{proof}
It is enough to prove that for $\BP$-almost every $\go\in\gO$, $(X_n(\go).o)_n$ is a $\gth_{\e,o}$-Cauchy sequence.\\
Let us consider the random variables $Y_n=d(X_{n-1}.o,X_n.o)=|H_n|_o$ for $n\ge1$.
These are real independent variables with distribution  $\gs\in\tprob(\BR_+)$ given by cumulative function $f_\gs(t)=\gs((-\infty,t])=\BP(|X_1|_o\le t)$ for $t\in \BR$.\\
Since $\mu$ has a first moment observe that:
\begin{align*}
\BE(Y_n)&=\int_{\BR_+}t\,d\gs(t)\\
&=\int_{\BR_+}\gs(t\ge a)\,d\gl(a)=\int_{\BR_+}\BP(|H_1|_o\ge a)\,d\gl(a)\\
&=\int_{\gO}|H_1|_o(\go)\,\BP(\go)=\sum_{g\in\Ga}\mu(g)|g.o|_o
\end{align*}
is finite and non-zero because $\mu$ is admissible.\\
The law of large numbers implies that for $\BP$-almost all $\go\in\gO$, $S_n(\go)=\frac{1}{n}\sum_{k=1}^nY_k(\go)\rightarrow \BE(|X_1|_o)$ and therefore $\frac{1}{n}Y_n\rightarrow 0$ $\BP$-almost surely.\\
On the other hand, since $\mu$ is non-elementary, $\frac{1}{n}|X_n|_o$ converges $\BP$-almost surely to $\ell>0$  and thus $\lim\frac{1}{n}(X_n|X_{n+1})_o(\go)=\ell$.
It follows for all $m\le n$ large enough:
\begin{align*}
\gth_{\e,o}(X_m(\go),X_n(\go))&\le\sum_{k=m}^\infty\rho_{\e,o}(X_k(\go),X_{k+1}(\go))\\
&\le\sum_{k=m}^\infty e^{-\half\ell k}\rightarrow 0, \quad \text{when $n,m$ go to infinity}
\end{align*}
for $\BP$-almost all $\go\in\gO$.
\end{proof}

Let us investigate the properties of the hitting random variable $Z$ and make some observations that will be useful in Theorem \ref{thm:equi} proof.
%Let $\alpha:\Ga\times K\rightarrow K$ be a continuous action of $\Ga$ on a compact metric space $(K,d')$ and an arbitrary measure $\mu$.\\
%A measure $m$ on $K$ is called \textit{stationary} if $\mu*_\alpha m:=\alpha_*(\mu\otimes m)$ is equal to $m$.\\
%The \textit{Markov operator} associated to $\mu$ is the positive contraction $P_\mu$ on $(C^0(K),\|.\|_\infty)$ given by:
%$$P_\mu\vp(x)=\sum_{g\in\Ga}\mu(g)\vp(g^{-1}x)$$
%The adjoint Markov operator, $P_\mu^*$, on space $\tprob(K)$ fixes any $\mu$-stationary measure on $K$.
%If $\nu$ is a stationary measure on $K$, $P_\mu$ is called $\nu$-ergodic if every $P_\mu$ invariant function is $\nu$-almost surely constant.
%It is uniquely ergodic if moreover it has an unique stationary probability measure.\\
%Note that the uniqueness of the stationary measure implies ergodicity \cite{MR3560700}.

\begin{rem}\label{rem:invmarkov}
For all $k\ge 0$ fixed, $\gTh^k$ preserves essentially the set upon which $(X_n)_n$ converges to $Z$ and as a consequence $X_n\circ\gTh^k(\go)\rightarrow Z\circ\gTh^k(\go)$.
On the other hand
$$(X_n\circ\gTh^k(\go)|X_n(\go))_o\ge\half(|X_n|_o(\go)-|H_{n+1}\dots H_{n+k}|_o(\go))$$
for all $\go\in\gO$ and
$$\frac{1}{n}|H_{n+1}\dots H_{n+k}|_o(\go)\le\frac{1}{n}\sum_{i=1}^k|H_{n+i}|_o(\go)$$
converges to 0 when $n$ goes to infinity and $k$ is fixed $\BP$-almost surely by the law of large numbers.
It follows that $(X_n\circ\gTh^k)_n$ and $(X_n)_n$ converge to the same limit:
$$Z\circ\gTh^k(\go)=Z(\go)$$
for $\BP$-almost all $\go\in\gO$. In other words the random variable $Z$ is Markov shift invariant.

Moreover for all $\go\in\gO$ one has $X_n\circ\gTh^k(\go)=X_{n+k}(\go)=X_k(\go).X'_{n}(\go)$ where $X'_n=X_k^{-1}X_{n+k}$ is a $\mu$-distributed random walk starting at $e$ independent of $X_{k'}$ for all $k'\le k$.\\
Therefore $\BP$-almost surely:
$$Z(\go)=\lim_n X_n\circ\gTh^k(\go)=X_k(\go).\lim_{n'} X_n'(\go)=X_k(\go).Z_k(\go)$$
where $Z_k$ is independent of $X_{k'}$ for all $k'\le k$.
If $U$ denotes the Bernoulli shift at $e$ on $\gO$ for $k\ge0$ one has: 
$$X_n\circ U^k=X_k^{-1}X_n\circ\gTh^k\rightarrow Z_k$$
and since $U$ is $\BP$-measure preserving one deduce that the sequence of random variables: $(Z_k)_{k\ge0}$ have the same law.

As a consequence if $\nu$ is the law of $Z$ on $\dd X$ the above relation implies:%%%lapage
\begin{align*}
\nu(A)=\BP(Z\in A)&=\BP(X_kZ_k\in A)\\
&=\int_\gO\mathbbm{1}_A(X_kZ_k)(\go)\,\BP(\go)\\
&=\sum_g\mu^{*k}(g)\int_\gO\mathbbm{1}_A(gZ_k(\go))\,\BP(\go)\\
&=\sum_g\mu^{*k}(g)\nu(g^{-1}A)
\end{align*}
for all measurable set $A\subset \dd X$.
In other words $\nu$ is $\mu$-stationary.
\end{rem}

\begin{lem}\label{lem:proximal}
Let $\nu$ be a diffuse probability on $\dd X$ and $(g_n)_n$ a proper sequence of elements of $\Ga$.
Assume ${g_n}_*\nu$ converges to $\nu'$, then $\nu'$ is a Dirac mass.
Moreover if the limit probability satisfies $\nu'=\gd_\xi$ then $\lim_n g_n.o=\xi$.
\end{lem}
\begin{proof}
Since $(g_n)_n$ is proper, after extraction of a subsequence, one may assume that there exist $\xi_\pm\in\dd\Ga$ such that $g_n^{\pm1}.o\rightarrow\xi_\pm$.\\
Let $\eta\neq\xi_-$ be a boundary point distinct from $\xi_-$.
Then $(\eta|g_n^{-1}.o)_o$ is a bounded sequence and since $|g^{-1}_n|_o-(\eta|o)_{g^{-1}_n.o}\approx(\eta|g_n^{-1}.o)_o$, it follows that $(\eta|o)_{g^{-1}_n.o}=(g_n\eta|g_n.o)_o\rightarrow+\infty$. 
In other words $g_n\eta\rightarrow \xi_+$.
Because $\nu$ is diffuse for all $\vp\in C^0(\dd X)$ the dominated convergence theorem implies:
$$\int_{\dd\Ga}\vp(\eta)\,d{g_n}_*\nu=\int_{\dd\Ga\setminus\xi_-}\vp(g_n\eta)\,d\nu\rightarrow \int_{\dd\Ga\setminus\xi_-}\vp(\xi_+)\,d\nu=\vp(\xi_+)$$
Conversely if $\nu'$ is Dirac mass at $\xi$ then $\xi$ is the only possible cluster value of $(g_n)_n$ by the previous argument.
\end{proof}

\begin{prop}\label{prop:hitting}
Let $Z$ be the random limit of a $\mu$-random walk $(X_n)_n$ on $\Ga$.
The law of $Z$, $\nu$, called hitting measure on $\dd X$ given by:
$$\nu=\int_{\dd X}\gd_{Z(\go)}\BP(\go)$$
has no atoms and is the unique $\mu$-stationary probability on $\dd X$.\\
%In other words the operator $P_{\mu}$ on $C^0(\dd X)$ is uniquely ergodic.
\end{prop}
\begin{proof}
Assume $\nu$ is a $\mu$-stationary probability on $\dd X$ with an atom at $\xi\in\dd X$.
Let $\vp_n\in\ell^1(\Ga)$ for $n\ge0$ be the sequence of positive functions on $\Ga$ given by  $\vp_n(g)=\mu^{*n}(g)\nu(g^{-1}\xi)$.
Since $\nu$ is $\mu$-harmonic one has:
$$\sum_{g\in\Ga}\vp_n(g)=\nu(\xi)>0$$
for all $n\ge 0$.\\
On the other hand the support of $\mu$ generates an non-amenable subgroup of $\Ga$.
This implies that $\mu^{*n}(g)\lesssim \rho(\mu)^n$ with $\rho(\mu)<1$ by Kersten theorem and therefore $(\vp_n)_n$ converges pointwise to $0$.\\
Since $\vp_n(g)\le u(g)=\nu(g^{-1}\xi)$ with $u\in\ell^1(\Ga)$, the dominated convergence theorem implies that $\sum_g\vp_n(g)\rightarrow 0$ but $\sum_g\vp_n(g)=\nu(\xi)>0$ for all $n\ge0$. This is a contradiction and $\nu$ is necessarily diffuse on $\dd X$.\\

Since $\nu$ is diffuse Lemma \ref{lem:proximal} implies that $X_n(.)_*\nu$ converges $\BP$-almost surely to $\gd_{Z(.)}$.
The martingale convergence theorem implies:
$$\int_\gO X_n(\go)_*\nu \BP(\go)\rightarrow \int_\gO \gd_{Z(\go)}\BP(\go)$$
Indeed for $\vp\in C^0(\dd X)$, $\vp_n(\go)=\int_{\dd X}\vp(X_n(\go)\xi)d\nu(\xi)$ with $\go\in\gO$ is a bounded martingale on $\gO$ and thus converges in $L^1(\gO,\BP)$.\\
One the other hand $\nu$ is $\mu$-stationary and thus:
$$\int_\gO X_n(\go)_*\nu \BP(\go)=\mu^{*n}*\nu=\nu$$
\end{proof}

%%%%%%%%%%%%%%%%%%%%%%%%%%%%%%%%%%%%%%%%%%%%%%%%Autre 
%\subsubsection{Random walks and Markov operators}
%$ $\\
%Given a continuous action of $\Ga$ on a compact metric space $(K,d')$ and a probability measure, $\mu$, on $\Ga$.
%The \textit{Markov operator} is the positive contraction $P_\mu$ on $(C^0(K),\|.\|_\infty)$ given by:
%$$P_\mu\vp(x)=\sum_{g\in\Ga}\mu(g)\vp(g^{-1}x)$$
%The $P_\mu^*$ adjoint operator on space $\tprob(K)$ fixes any $\mu$-stationary measure on $K$.
%\begin{rem}
%More generally, using the same arguments, every probability measure on $\Ga$ with spectral gap defines an uniquely ergodic Markov operator on $\dd X$.
%\end{rem}

%%%%%%%%%%%%%%%%%%%%%%%%%%%%%%%%%%%%%%%%%%%%%%%%
%%%%%%%%%%%%%%%%%%%%%%%%%%%%%%%%%%%%%%%%%%%%%%%%%%%%%%%%%%%%%%%%%%
%%%%%%%%%%%%%%%%%%%%%%%%%%%%%%%%%%%%%%%%%%%%%%%%%%%%%%%%%%%%%%%%%%
%%%%%%%%%%%%%%%%%%%%%%%%%%%%%%%%%%%%%%%%%%%%%%%%%%%%%%%%%%%%%%%%%%
%%%%%%%MU IS ASSUMED TO SYMMETRIC NON AMENABLE WITH 1ST MOMENT%%%%%%%%%%%%%%%%%%%%%%%%%%%
%\section{Statements and proof of the main results}
\section {Probabilistic Margulis-Roblin equidistribution}\label{sec:equi}
The next theorem is a probabilistic analogue of Roblin equidistribution Theorem 4.1.1 of \cite{MR2057305} (see also \cite{MR2035655}).
In this section $\mu$ denote an arbitrary symmetric non-elementary probability measure on $\Ga$ with 1st moment.
\begin{thm}\label{thm:equi}
Let $\vp_1,\vp_2\in C^0(\dd X)$ be two continuous functions on $\dd X$, $(X_n)_n$ a $\mu$-random walk starting at $o\in X$ with $\mu$ as above and $\nu$ the unique $\mu$-stationary measure on $\dd X$.\\
Then the following equidistribution holds:
$$\BE^\mu[\vp_1(\hat{X_n.o})\vp_2(\hat{X_n^{-1}.o})]=\sum\vp_1(\hat{g.o})\vp_2(\hat{g^{-1}.o})\mu^{*n}(g)\xrightarrow{n\rightarrow+\infty} \int_{\dd X}\vp_1\,d\nu.\int_{\dd X}\vp_2\,d\nu.$$
Moreover the boundary retraction might be chosen differently for $\vp_1$ and $\vp_2$.
\end{thm}

%In other words if we consider:
%$$B: \Ga\rightarrow \dd X\times\dd X, \quad B(g)=(\hat{g.o},\hat{g^{-1}.o})$$
%the sequence of probabilities $(B_*\mu^{*n})_n$ on $\dd X\times \dd X$ converges weakly to $\nu\otimes{\nu}$.\\
%%%%%%%%%%%%%%%%%%%%%%%%%%%%%%%%%%%%%%%%%%%%%%%%%%%%%%%%%%%
\begin{cor}\label{cor:markov}
Given $\vp\in C^0(\dd X\times\dd X)$ a continuous function on $\dd X\times\dd X$, the following equidistribution holds:
$$\BE^\mu[\Phi(X_n\xi,X_n^{-1}\eta)]=\sum_{g\in\Ga}\Phi(g\xi,g^{-1}\eta)\mu^{*n}(g)\xrightarrow{n\rightarrow+\infty} \int_{\dd X\times\dd X}\Phi\,d\nu\otimes\nu.$$
%In particular $P_{(2)}$ is $\nu\otimes\nu$ uniquely ergodic on $\dd X\otimes \dd X$.
\end{cor}
\begin{proof}
Observe that:
$$C^0(\dd X)\otimes_{\text{alg}} C^0(\dd X)\subset C^0(\dd X)\hat{\otimes}_\e C^0(\dd X)=C^0(\dd X\times\dd X)$$
where $\otimes_{\text{alg}}$ and $\hat{\otimes}_\e$ stand respectively for the algebraic tensor product and the injective one \cite{MR1888309}.
If $J:C^0(\dd X\times\dd X)\rightarrow C^0(\dd X\times\dd X)^{**}$ denote the topological bidual injection one has:
$$\ol{J[C^0(\dd X)\otimes_{\text{alg}} C^0(\dd X)]}^{\text{w}}=C^0(\dd X\times\dd X)^{**}$$
where $\ol{*}^{\text{w}}$ stands for the weak closer on $C^0(\dd X\times\dd X)^{**}$.
Using the fact that probability measures on $C^0(\dd X\times\dd X)$ are contractions one deduce the extension of Theorem \ref{thm:equi} for arbitrary function on $\dd X\times\dd X$ and choice of boundary retractions.

Take $B:\Ga=X\to\dd X$ given by $B(g)=g.\xi_0$ for some fixed $\xi_0\in\dd X$. 
Since $\Ga$ acts geometrically on $X$, $B$ extends to a boundary retraction on $\ol{X}$.
The corollary follows from Theorem \ref{thm:equi} applied to boundary retractions of this type.
%Given any $P_{(2)}$-stationary probability measure $\nu_{(2)}$ on $\dd X\times\dd X$ and continuous function $\Phi\in C^0(\dd X\times\dd X)$ the dominated convergence theorem implies:
%$$\int_{\dd X\times\dd X}\Phi d\nu_{(2)}=\int_{\dd X\times\dd X}P^n_{(2)}\Phi d\nu_{(2)}\rightarrow \int_{\dd X\times\dd X}\Phi d\nu\otimes\nu$$
%which proves the unique ergodicity.
\end{proof}

%Another consequence is $P_{(2)}$-unique ergodicity of the measure $\nu\otimes\nu$.
%Indeed given any $P_G$-invariant probability measure, $\nu'$, then for all $\Psi\in C^0(\dd X\times\dd X)$, the dominated convergence theorem implies:
%$$\int_{\dd X\times\dd X}\Psi d\nu'=\int_{\dd X\times\dd X}P_G^n\Psi d\nu'\rightarrow \int_{\dd X\times\dd X}\Psi d\nu\otimes\nu$$

\begin{proof}[Proof of Theorem \ref{thm:equi}]
Let us denote $I_n=\BE^\mu[\vp_1(\hat{X_n^{-1}.o})\vp_2(\hat{X_n.o})]$ for $n\ge0$ and $I_+=\limsup_n I_n$, which is bounded above by $\|\vp_1\|_\infty\|\vp_2\|_\infty$.\\
It is enough to prove:
$$I_+\le \int_{\dd X}\vp_1\,d\nu.\int_{\dd X}\vp_2\,d\nu$$
Indeed, exchanging $\vp_1$ for $-\vp_1$ leads to: 
$$\int_{\dd X} \vp_1d\nu.\int_{\dd X} \vp_2d\nu\le \liminf_n I_n$$
and concludes the proof of Theorem \ref{thm:equi}.\\

Let us start by introducing the maps:
$$F_{n,k}:\gO\rightarrow \BR_+,\quad \go\mapsto (X_n(\go).o|X_n\circ\gTh^k(\go).o)_{o}$$
and 
$$F'_{n,k}:\gO\rightarrow \BR_+,\quad \go\mapsto (X^{-1}_n\circ U^k(\go).o|X_n^{-1}\circ\gTh^k(\go).o)_{o}$$
for $n,k\ge0$ where $U$ and $\gTh$ denote respectively the Bernoulli and the Markov shift on $\gO$.\\
Observe that:
\begin{align*}
&|F_{n,k}-\ell.n|\\
&=\half|(|X_n.{o}|_{o}-\ell.n)+(|X_{n+k}.{o}|_{o}-\ell.(n+k))-(|H_{n+1}\dots H_{n+k}.{o}|_{o}-\ell.k)|;\\
&|F'_{n,k}-\ell.n|\\
&=\half|(|H_{n+k}^{-1}\dots H_{k+1}^{-1}.{o}|_{o}-\ell.n)+(|X_{n+k}^{-1}.{o}|_{o}-\ell.(n+k))-(|H_{k}^{-1}\dots H_{1}^{-1}.{o}|_{o}-\ell.k)|
\end{align*}
for all $n,k\ge0$.
The subadditive ergodic theorem \cite{MR704553} guarantees that $\frac{|X_n|_o}{n}\rightarrow \ell$ pointwise and in $L^1(\gO,\BP)$ where $\ell>0$ denotes the drift of $\mu$, that is non-zero since the support of $\mu$ generates a non-amenable subgroup in $\Ga$.
Moreover following Remark \ref{rem:invmarkov} the 1st moment assumption implies that $\frac{F_{n,k}}{n},\frac{F_{n,k}'}{n}\rightarrow\ell$ pointwise and in $L^1(\gO,\BP)$, when $k$ is fixed and $n$ goes to infinity.
%Note that the convergence of $\frac{|H^{(-1)}_{n+1}\dots H^{(-1)}_{n+k}o|_o}{n}$ to $0$ when $k$ is fixed and $n$ goes to infinity is a consequence of the law of large number since $\mu$ is assumed to have a 1st moment (cf. Proposition \ref{lem:asconv} proof).\\

On the other hand since $(\vp_2(\hat{X_n^{-1}.o}))_n$ defines a bounded sequence in $L^\infty(\gO,\BP)=L^1(\gO,\BP)^*$ one can assume that $\vp_2(\hat{X_n^{-1}.o})$ converges to some function $\vp_{-\infty}\in L^\infty(\gO,\BP)$ for the weak-$*$ topology $\gs(L^\infty,L^1)$.\\

We are going to prove the following inequality:
$$I_+\le \int_\gO\vp_1(Z\circ U^{k}(\go))\vp_{-\infty}(\go)\BP(\go)$$

Let us fixe $k\in\BN$ and $\e>0$.
Using Egoroff theorem one can find $B\subset\gO$ with $\BP(B)\le \e$ such that $(\frac{F_{n,k}}{n})_n$ and $(\frac{F'_{n,k}}{n})_n$ converge uniformly to $\ell$ on $A=\gO\setminus B$.\,
% when $n$ goes to infinity, $B_{n+1}\subset B_n$ and $M=M(n)\ge n$ such that for all $m\ge M$
%$$\gth_{o,\e}(X_m(\go).{o},X_m\circ\gTh^k(\go).{o})\le e^{-\half\ell n}$$%modifié
%respectively,
%$$\gth_{o,\e}(X^{-1}_m\circ U^k(\go).{o},X_m^{-1}\circ\gTh^k(\go).{o})\le e^{-\half\ell n}$$%modifié
%for all $\go\in A_n$.\\
Because $\hat{\vp_i}$, $i=1,2$, are uniformly continuous on the compact $\overline{X}$ there exists $\eta\le\e$ such that for all $x,y\in\ol{X}$ with $\gth_{o,\e}(x,y)\le\eta$, $|\hat{\vp_i}(x)-\hat{\vp_i}(y)|\le\e$.\\

Take $n_0$ large enough such that $e^{-\half\ell n_0}\le\eta$ and for all $n\ge n_0$,
\begin{equation}\label{eq:dist}
|\frac{F_{n,k}(\go)}{n}-\ell|+|\frac{F'_{n,k}(\go)}{n}-\ell|\le \half\ell
\end{equation}
%\begin{equation}\label{eq:simple}
%|\BE^\mu[\hat{\vp}_i(X_{n}.{o_i})]-\int_{\dd X} \vp_i(\xi)d\nu(\xi)|\le \e
%\end{equation}
for all $\go\in\gO\setminus B$,
$$|\BE^\mu[\vp_1(Z\circ U^{k})\hat{\vp}_2(X^{-1}_{n}.{o})]-\int_{\gO} \vp_1(Z\circ U^{k}(\go))\vp_{-\infty}(\go)\BP(\go)|\le \e$$
and
$$\int_\gO|\vp_1(Z(\go))-\hat{\vp}_1(X_{n}(\go).{o})|\BP(\go)\le\e.$$

%for $i=1,2$ and all $n\ge n_0$.\\
In particular Equation \ref{eq:dist} implies:
$$\gth_{o,\e}(X_n(\go).{o},X_n\circ\gTh^k(\go).{o}),\,\gth_{o,\e}(X^{-1}_n\circ U^k(\go).{o},X_n^{-1}\circ\gTh^k(\go).{o})\le \eta$$
for $n\ge n_0$ and $\go\in\gO\setminus B$.\\
%The existence of such $n_0$ follows by a dominated convergence argument applied to the $(\hat{\vp}_i(X_{n}))_n$ which converge pointwise to $\vp_i(Z)$ for $i=1,2$ where $Z$ is $\nu$-distributed.\\

For $n\ge n_0$ one has:
$$I_+-\e\le\BE^\mu[\hat{\vp}_1(X_{n+{k}}^{-1}.{o})\hat{\vp}_2(X_{n+{k}}.{o})]=\BE^\mu[\hat{\vp}_1(X_{n}^{-1}\circ\gTh^{k}(\go).{o})\hat{\vp}_2(X_{n}\circ\gTh^{k}(\go).{o})]$$
Observe that:
\begin{align*}
%\BE^\mu[&\hat{\vp}_1(X_{n_1}^{-1}\circ\gTh^{k_1}(\go).{o})\hat{\vp}_2(X_{n_1}\circ\gTh^{k_1}(\go).{o})]\\
&\int_\gO\hat{\vp}_1(X_{n}^{-1}\circ\gTh^{k}(\go).{o})\hat{\vp}_2(X_{n}\circ\gTh^{k}(\go).{o})\BP(\go)=\int_\gO\hat{\vp}_1(X_{n}^{-1}\circ U^{k}(\go).{o})\hat{\vp}_2(X_{n}(\go).{o})\BP(\go)\\
&+\int_{\gO=\gO\setminus B\cup B}(\hat{\vp}_1(X_{n}^{-1}\circ\gTh^{k}(\go).{o})-\hat{\vp}_1(X_{n}^{-1}\circ U^{k}(\go).{o}))\hat{\vp}_2(X_{n}\circ\gTh^{k}(\go).{o})\BP(\go)\\
&+\int_{\gO=\gO\setminus B\cup B}\hat{\vp}_1(X_{n}^{-1}\circ U^{k}(\go).{o})(\hat{\vp}_2(X_{n}\circ\gTh^{k}(\go).{o})-\hat{\vp}_2(X_{n}(\go).{o}))\BP(\go)\\
&\le\int_\gO\hat{\vp}_1(X_{n}^{-1}\circ U^{k}(\go).{o})\hat{\vp}_2(X_{n}(\go).{o})\BP(\go)+4\BP(B)\|\vp_1\|_\infty\|\vp_2\|_\infty+2\e\max_{i=1,2}\|\vp_i\|_\infty
\end{align*}

In other words:
$$I_++O(\e)\le \int_\gO\hat{\vp}_1(X_{n}\circ U^{k}(\go).{o})\hat{\vp}_2(X_{n}^{-1}(\go).{o})\BP(\go)$$

On the other hand, using the inequality $\|\vp_{-\infty}\|_\infty\le\|\vp_2\|_\infty$ one obtain:
\begin{align*}
&|\int_\gO\hat{\vp}_1(X_{n}\circ U^{k}(\go).{o})\hat{\vp}_2(X_{n}^{-1}(\go).{o})\BP(\go)-\int_\gO\vp_1(Z\circ U^{k}(\go))\vp_{-\infty}(\go)\BP(\go)|\\
&\le|\int_\gO\hat{\vp}_1(X_{n}\circ U^{k}(\go).{o})[\hat{\vp}_2(X_{n}^{-1}(\go).{o})-\vp_{-\infty}(\go)]\BP(\go)|\\
&+|\int_\gO[\vp_1(Z\circ U^{k}(\go))-\hat{\vp}_1(X_{n}\circ U^{k}(\go).{o})]\vp_{-\infty}(\go)\BP(\go)|\\
&\le|\int_\gO\hat{\vp}_1(X_{n}\circ U^{k}(\go).{o})[\hat{\vp}_2(X_{n}^{-1}(\go).{o})-\vp_{-\infty}(\go)]\BP(\go)|\\
&+\|\vp_2\|_\infty\int_\gO|\vp_1(Z(\go))-\hat{\vp}_1(X_{n}(\go).{o})|\BP(\go)\le \e(1+\|\vp_2\|_\infty)\\
\end{align*}
%%HERE
and since $\e$ can be chosen arbitrary small it follows that:
$$I_+\le \int_\gO\vp_1(Z\circ U^{k}(\go))\vp_{-\infty}(\go)\BP(\go)$$
for all $k\in\BN$.\\

Eventually observe that the weak convergence of $(\vp_2(X_n^{-1}.o))_n$ to $\vp_{-\infty}$ implies:
$$\BE^\mu(\vp_2(X_n^{-1}.o))\rightarrow \int_\gO\vp_{-\infty}(\go)\BP(\go)$$ 
and the symmetry of $\mu$ gives:
$$\BE^\mu(\vp_2(X_n^{-1}.o))=\BE^\mu(\vp_2(X_n.o))\rightarrow \int_{\dd X}\vp_{2}d\nu$$
This imposes:
$$\int_\gO\vp_{-\infty}(\go)\BP(\go)=\int_{\dd X}\vp_{2}d\nu$$

Since the Bernouilli shift $U$ is mixing one deduce:
%The last step consists to use the mixing property of  to deduce:
\begin{align*}
I_+\le \int_\gO\vp_1(Z\circ U^{k}(\go))\vp_{-\infty}(\go)\BP(\go)\xrightarrow{k\rightarrow+\infty}&\int_{\dd X}\vp_1d\nu\int_\gO\vp_{-\infty}(\go)\BP(\go)\\
&=\int_{\dd X}\vp_1d\nu\int_{\dd X}\vp_2d\nu
\end{align*}
%Since $k_1$ was chosen larger than ${n_1}$ it follows that $X_{n_1}^{-1}\circ U^{k_1}$ and $X_{n_1}$ are independent and
%\begin{align*}
%\BE^\mu[\hat{\vp}_1(X_{n_1}^{-1}\circ U^{k_1}.{o})\hat{\vp}_2(X_{n_1}.{o})]&=\BE^\mu[\hat{\vp}_1(X_{n_1}^{-1}\circ U^{k_1}.{o})]\BE^\mu[\hat{\vp}_2(X_{n_1}.{o})]\\
%&=\BE^\mu[\hat{\vp}_1(X_{n_1}.{o})]\BE^\mu[\hat{\vp}_2(X_{n_1}.{o})]
%\end{align*}
%the last equality follows from the fact that the Bernoulli shift is a measure preserving transformation and $\mu$ is a symmetric measure.

%The choice of ${n_1}$ implies:
%$$\BE^\mu[\hat{\vp}_1(X_{n_1}.{o})]\BE^\mu[\hat{\vp}_2(X_{n_1}.{o})]\le \int_{\dd X} \vp_1(\xi)d\nu(\xi)\int_{\dd X} \vp_2(\xi)d\nu(\xi)+2\e\max_{i=1,2}\|\vp_i\|_\infty$$
%and thus
%$$I_++O(\e)+o_{n_1}(1)\le\int_{\dd X} \vp_1d\nu.\int_{\dd X} \vp_2d\nu$$
%Since $\e$ and ${n_0}$ can be chosen respectively arbitrary small and large one deduce 
%$$I_+\le\int_{\dd X} \vp_1d\nu.\int_{\dd X} \vp_2d\nu$$
%Exchanging $\vp_1$ for $-\vp_1$ a similar argument leads to: 
%$$\int_{\dd X} \vp_1d\nu.\int_{\dd X} \vp_2d\nu\le \liminf_n I_n$$
%and this concludes the proof of Theorem \ref{thm:equi}
\end{proof}

%%%%%%%%%%%%%%%%%%%%%%%%%%%%%%%%%%%%%%%%%%%%%%%%%%%%%%%%%%%%
%%%%%%%%%%%%%%%%%%%%%%%%%%%%%%%%%%%%%%%%%%%%%%%%%%%%%%%%%%%%
%%%%%%%%%%%%%%%NEW SECTION%%%%%%%%%%%%%%%%%%%%%%%%%%%%%%%%%%%%
\section{A probabilistic approach of the irreducibility}\label{sec:main}%%SUBSECTION%%%%%%%%%%%TO DO IN A MORE GENERAL SETUP
%Let $\Ga$ be a discrete group acting geometrically on $X=\Ga$ endowed with $d\in\cD(\Ga)$ a Green metric associated a admissible finitely supported probability $\mu_0$.\\
Let $\nu$ be a Ahlfors-regular quasiconformal measure on $\dd X$ for the action of $\Ga$ on $(X,d)$. 
The principal example in our framework are stationary measures on $\dd X$ for finitely supported symmetric admissible probabilities on $\Ga$.
The associated \textit{boundary representation} of $\Ga$ is defined as:
$$\pi_o:\Ga\rightarrow \mathcal{U}[L^2[\dd X,\nu]],\quad \pi_o(g)\vp(\gx)=\sqrt{r_{o}(g^{-1},\xi)}\vp(g^{-1}\gx)$$
where $r_{o}$ denotes the Radon-Nikodym derivative cocycle given by the formula:
$$r_{\nu}(g,\xi)=\frac{dg^{-1}_*\nu}{d\nu}(\xi)\asymp e^{\e D_\e\gb_\xi(g^{-1}.o)}$$
with $g\in\Ga$, $\xi\in\dd X$.\\
%If $\nu$ denotes the unique $\mu_0$-stationary measure on $\dd X$ the \textit{boundary representation} on $(\Ga,\mu_0)$ is defined as:
%$$\pi_o:\Ga\rightarrow \mathcal{U}[L^2[\dd X,\nu]],\quad \pi_o(g)\vp(\gx)=\sqrt{r_{o}(g^{-1},\xi)}\vp(g^{-1}\gx)$$

%Using the identification established in \cite{MR2919980} this stationary measure can be seen as the Patterson Sullivan measure of the hyperbolic Green distance $d$ and following holds:
%$$\frac{dg^{-1}_*\nu}{d\nu}=e^{-\gd_\Ga\beta_\gx(g^{-1}.o)}$$
%with $g\in\Ga$, $\xi\in\dd X$ and $\gd_\Ga=1$ the critical exponent of $\Ga$ with respect to $d$.\\

The \textit{Harish-Chandra function} on $\Ga$ is defined as the matrix coefficient given by: 
$$\Xi_o(g)=(\pi_o(g){\bf 1}_{\dd\Ga}|{\bf 1}_{\dd\Ga})=\|\sqrt{r_{o}(g^{-1},.)}\|_{L^1(\dd X,\nu)}=\Xi_o(g^{-1})$$
for $g\in \Ga$.\\ %, which controls the behavior of matrix coefficients at infinity cf. \ref{}.\\
In the rest one denote: 
$$\ti{\pi}_o(g)=\frac{\pi_o(g)}{\Xi_o(g)}$$ 
the renormalization of the representation $\pi_o$ by the Harish-Chandra function $\Xi_o$.
% which, as it is shown below, satisfies better analytic properties.\\

Let $\mathcal{P}_{{\bf1}_{\dd X}}$ be the one dimensional projector on $L^2[\dd X,\nu]$ given by:
$$\mathcal{P}_{{\bf1}_{\dd X}}(\vp)=(\vp|{\bf1}_{\dd X}){\bf1}_{\dd X}$$
for $\vp\in L^2[\dd X,\nu]$.

The principal result of this section is:
\begin{thm}\label{thm:flipap}
Let $\mu\in \tprob(\Ga)$ be a finitely supported, symmetric probability on $\Ga$ and $\nu$ its associated stationary measure on $\dd X$.\\
Denote $\mathcal{P}_n\in \mathcal{B}(L^2[\dd X,\nu])$ the operator defined as:
$$\mathcal{P}_n=\BE^\mu[\widetilde{\pi}_o(X_n)]=\sum_{g\in\Ga}\frac{\pi_o(g)}{\Xi_o(g)}\mu^{*n}(g).$$
Then $(\mathcal{P}_n)_n$ defines an uniformly bounded sequence of operators which converges to the one dimensional projector $\mathcal{P}_{{\bf1}_{\dd X}}$ for the weak-* operator topology, in other words:
$$\sum_{g\in\Ga}\frac{(\pi_o(g)\vp|\psi)}{\Xi_o(g)}\mu^{*n}(g)\xrightarrow{n\rightarrow+\infty}\int_{\dd X}\vp d\nu\int_{\dd X}\psi d\nu$$
for all $\vp,\psi\in L^2[\dd X,\nu]$.
\end{thm}

Before we go to the proof of Theorem \ref{thm:flipap} let us give a direct consequence:
\begin{cor}\label{cor:irr}
Let $\mu\in \tprob(\Ga)$ be a finitely supported, symmetric probability on $\Ga$ and $\nu$ its associated stationary measure on $\dd X$.\\
Then the associated boundary representation $(\Ga,\pi_o,L^2[\dd X,\nu])$ is irreducible.
\end{cor}
\begin{proof}
We start by proving that the vector ${\bf1}_{\dd X}$ is cyclic for $(\Ga,\pi_o,L^2[\dd X,\nu])$.
Indeed assume $\psi\in L^2[\dd X,\nu]\ominus \ol{\text{span}\{\pi(g){\bf1}_{\dd X}|\,g\in\Ga\}}$, then Proposition \ref{prop:coef} implies:
\begin{align*}
\BE^\mu&[\hat{w}(X_n)(\tilde{\pi}_o(X_n){\bf1}_{\dd X}|\psi)_{L^2}]\\
&=\sum_{g\in\Ga}\hat{w}(g)(\tilde{\pi}_o(g){\bf1}_{\dd X}|\psi)_{L^2}\xrightarrow{n\rightarrow+\infty} \int_{\dd X}\psi(\xi)w(\xi)d\nu(\xi)
\end{align*}
for all $w\in \lips(\dd X,d_{o,\e})$.
Using the density of Lipschitz functions together with the fact that $\BE^\mu[\hat{w}(X_n)(\tilde{\pi}_o(X_n){\bf1}_{\dd X}|\psi)_{L^2}]=0$ for all $n$ it follows that $\psi=0$.\\

Let $V_0$ be a non-trivial subrepresentation of $(\Ga,\pi_o,L^2[\dd X,\nu])$.
Then there exists $\psi\in V_0$ such that $({\bf1}_{\dd X}|\psi_0)_{L^2}\neq0$.
Otherwise using the invariance of this subrepresentation we would have that $V_0$ is orthogonal to $\ol{\text{span}\{\pi(g){\bf1}_{\dd X}|\,g\in\Ga\}}=L^2[\dd X,\nu]$.

Finally since $\mathcal{P}_{{\bf1}_{\dd X}}$ belongs to the von Neumann algebra of $\pi_o$ it follows that $\mathcal{P}_{{\bf1}_{\dd X}}(\psi_0)=\gl_0{{\bf1}_{\dd X}}\in V_0$ with $\gl_0\neq0$ and therefore $V_0=L^2[\dd X,\nu]$.
\end{proof}
%\bigskip
%%%%%%%%%%%%%%%%%%%%%%%%%%DEBUT DE LA PREUVE%%%%%%%%%%%%%%%%%%%%%%%
%%%%%%%%%%%%%%%%%%%%%%%%%%DEBUT DE LA PREUVE%%%%%%%%%%%%%%%%%%%%%%%
%%%%%%%%%%%%%%%%%%%%%%%%%%DEBUT DE LA PREUVE%%%%%%%%%%%%%%%%%%%%%%%

\subsubsection{Proximal phenomenon in $L^2$ and boundary representations.}
$ $\\
A fundamental estimate of the Harish-Chandra is given by the following lemma:
\begin{lem}[\cite{MR3622235} \cite{Garncarek:aa}]\label{lem:est}
Let $\Xi_o$ be the Harish-Chandra function on $\Ga$ associated to its geometric action on $(X,d)$.
Then for all $g\in\Ga$:
$$\Xi_o(g)\asymp(1+|g.o|_o)e^{-\half\e D_\e|g.o|_o}.$$
where $D_\e$ is the Hausdorff dimension of $\nu$.% critical exponent of $\Ga$.
\end{lem}
%Let us insist on the fact that this formula only holds in the context of convex cocompact hyperbolic groups.\\

As explained in the following lemma the sequence of absolutely continuous probabilities with respect to $\nu$ given by:
$$u_gd\nu=\frac{\sqrt{r_{o}(g^{-1},.)}}{\Xi_o(g)}d\nu\asymp\frac{e^{\e D_\e(g.o|\cdot)_o}}{(1+|g.o|_o)}d\nu$$
on $(\dd X,\nu)$ gives an approximation of the Dirac mass in the sense:
$$(\vp|u_g)_{L^2}=\int\vp(\gx)\frac{e^{\e D_\e(g.o|\gx)_o}}{(1+|g.o|_o)}\,d\nu(\gx)\rightarrow \vp(\xi_0)$$
when $g$ goes to $\xi_0$.\\
Moreover this convergence is controlled by the length of $g$:
\begin{lem}\cite{Garncarek:aa}\label{lem:l2proximal}
Let $\Psi\in\lips(\dd X\times \dd X,d)$ be a Lipschitz function on $\dd X\times \dd X$ for the $\ell^1$-product distance:
$d((\xi,\eta),(\xi',\eta'))=d_{o,\e}(\xi,\xi')+d_{o,\e}(\eta,\eta')$ for $\xi,\xi',\eta,\eta'\in\dd X$.
Assume $\e>0$ small enough such that the Hausdorff dimension, $D_\e$, of $(\dd X,d_{o,\e})$ is strictly greater than $1$.
Then:
$$|(\Psi|u_g\otimes u_{g^{-1}})_{L^2}-\Psi(\hat{g},\widecheck{g})|\le\frac{2\gl(\Psi)}{(1+|g.o|_o)^{1/{D_\e}}}$$
for all $g\in\Ga$, where $\gl(\Psi)$ is the Lipschitz constant of $\Psi$ and $\widecheck{g}=\hat{g^{-1}.o}$.
\end{lem}
%\begin{proof}
%Using the fact 
%$$\gb(o,g.o;\xi)=2(\xi|g.o)_o-|g.o|_o\preceq2(\xi|\hat{g.o})_o-|g.o|_o$$
%for all $g\in\Ga$, $\xi\in \dd X$ together with Lemma \ref{lem:est} we deduce for any $0<r<\text{Diam($\dd X$)}$:
%\begin{align*}
%|(\vp|u_g)-\vp(\hat{g})|&=|(\vp-\vp(\hat{g})|u_g)|\\
%&\le\gl(\vp)|\int_{\dd X=B_{\dd X}(\hat{g},r)\cup B^c_{\dd X}(\hat{g},r)}d(\hat{g},\xi)u_g(\xi)d\nu(\xi)|\\
%&\preceq\gl(\vp)[r+|\int_{B^c_{\dd X}(\hat{g},r)}\frac{d(\hat{g},\xi)^{1-D_\e}}{1+|g.o|_o}d\nu(\xi)|]\\
%&\le\gl(\vp)[r+\frac{r^{1-D_\e}}{1+|g.o|_o}]\\
%\end{align*}
%The lemma follows by taking $r=(1+|g.o|_o)^{-1/{D_\e}}$.
%\end{proof}

\begin{proof}
Using the fact :
$$\gb(o,g.o;\xi)=2(\xi|g.o)_o-|g.o|_o\preceq2(\xi|\hat{g.o})_o-|g.o|_o$$
for all $g\in\Ga$, $\xi\in \dd X$ together with Lemma \ref{lem:est} we deduce for any $0<r<\text{Diam($\dd X$)}$:
\begin{align*}
&|(\Psi|u_g\otimes u_{g^{-1}})_{L^2}-\Psi(\hat{g},\widecheck{g})|=|(\Psi-\Psi(\hat{g},\widecheck{g})|u_g\otimes u_{g^{-1}})_{L^2}|\\
&\le\gl(\Psi)|\int_{\dd X}[d_{o,\e}(\hat{g},\xi)+d_{o,\e}(\widecheck{g},\eta)]u_g(\xi)u_{g^{-1}}(\eta)d\nu(\xi)d\nu(\eta)|\\
&\preceq\gl(\Psi)[2r+|\int_{B^c_{\dd X}(\hat{g};r)}\frac{d_{o,\e}(\hat{g},\xi)^{1-D_\e}}{1+|g.o|_o}d\nu(\xi)|+|\int_{B^c_{\dd X}(\widecheck{g};r)}\frac{d_{o,\e}(\widecheck{g},\eta)^{1-D_\e}}{1+|g.o|_o}d\nu(\eta)|]\\
&\le2\gl(\Psi)[r+\frac{r^{1-D_\e}}{1+|g.o|_o}]\\
\end{align*}
where the second inequality is obtained by the decomposition $\dd X=B_{\dd X}((\hat{g},\widecheck{g});r)\cup B^c_{\dd X}((\hat{g},\widecheck{g});r)$.
The lemma follows by taking $r=(1+|g.o|_o)^{-1/{D_\e}}$.
\end{proof}

\begin{rem}
%Note that the assumption on $\e$ does not affect the measure $\nu$ nor the boundary representation considered.\\
Since $\nu$ is Ahlfors-regular and the Lebesgue differentiation theorem guarantees the density of Lipschitz functions in $L^2[\dd X,\nu]$.
\end{rem}

%\begin{rem}
%Let $(g_n)_n$ be a sequence of points in $\Ga$ such that $g_n^{-1}.o\rightarrow \xi\in\dd X$.
%Then Lemma \ref{lem:proximal} implies:
%\begin{align*}
%(\vp|r_{\nu}(g_n^{-1}))&=\int_{\dd X}\vp(\eta)r_{\nu}(g_n^{-1})(\eta)d{\nu}\\
%&=\int_{\dd X}\vp(\eta)dg_n^{-1}{\nu}\rightarrow \vp(\xi).
%\end{align*}
%In particular Lemma \ref{lem:l2proximal} can be observed as a $L^2$-analogue of this phenomenon.
%\end{rem}

\begin{prop}\label{prop:coef}
Given $w,\vp,\psi\in\lips(\dd X,d_{o,\e})$, three Lipschitz functions on $\dd X$, the following holds:
$$\BE^\mu[|w(\hat{X_n.o})(\ti{\pi}_o(X_n)\vp|\psi)-w(\hat{X_n.o})\vp(\hat{X_n^{-1}.o}).\psi(\hat{X_n.o})|]\xrightarrow{n\rightarrow+\infty} 0$$
\end{prop}

This Proposition \ref{prop:coef} together with Theorem \ref{thm:equi} implies:
\begin{cor}
Given any Lipschitz functions $w,\vp, \psi\in\lips(\dd X,d_{o,\e})$ one has:
$$\BE^\mu[w(\hat{X_n.o})(\ti{\pi}_o(X_n)\vp|\psi)]\rightarrow \int_{\dd X}\vp\,d{\nu}.\int_{\dd X}w.\psi\,d{\nu}$$
when $n$ goes to infinity.
\end{cor}

\begin{proof}[Proof Proposition \ref{prop:coef}]
Let $w,\vp,\psi\in\lips(\dd X,d_{o,\e})$ be three Lipschitz functions.

Take $0<s<1$ such that $\rho_\mu\gd_\Ga^s<1$, where $\rho_\mu<1$ is the spectral radius of $\mu$ and $\gd_\Ga$ the growth rate of $\Ga$.\,% or the critical exponent of $\Ga$ when $\Ga$ acts geometrically on $(X,d)$.
Observe that:
\begin{align*}
\BP(|X_n|_o\le sn)&=\mu^{*n}(B_X(o,sn))\\
&\le \max\{\mu^{*n}(e),\mu^{*(n-1)}(e)\}|B_X(o,sn)\cap\Ga.o|\\
&\preceq \rho_\mu^n\gd_\Ga^{sn}=(\rho_\mu\gd_\Ga^s)^n
\end{align*}
for all $n$.\\
On the other hand:
\begin{align*}
&\sum_{g;\,|g.o|_o\ge sn}|w(\hat{g.o})(\ti{\pi}_o(g)\vp|\psi)-w(\hat{g.o})\vp(\hat{g^{-1}.o}).\psi(\hat{g.o})|\mu^{*n}(g)\\
&=\sum_{g;\,|g.o|_o\ge sn}|w(\hat{g.o})(\ti{\pi}_o(g)\vp|\psi-\psi(\hat{g.o}){\bf1}_{\dd X})\\
&\quad-w(\hat{g.o})(\vp(\hat{g^{-1}.o}){\bf1}_{\dd X}-\vp|\ti{\pi}(g^{-1}){\bf1}_{\dd X}).\psi(\hat{g.o})|.\mu^{*n}(g)\\
&\le \|w\|_\infty\|\vp\|_\infty\sum_{g;\,|g.o|_o\ge sn}|(\psi-\psi(\hat{g.o}){\bf1}_{\dd X}|u_g)|.\mu^{*n}(g)\\
&+\|w\|_\infty\|\psi\|_\infty\sum_{g;\,|g.o|_o\ge sn}|(\vp(\hat{g^{-1}.o}){\bf1}_{\dd X}-\vp|u_{g^{-1}})|.\mu^{*n}(g)\\
&\le \|w\|_\infty\frac{\|\vp\|_\infty\gl(\psi)+\|\psi\|_\infty\gl(\vp)}{(1+ sn)^{1/D_\e}}
\end{align*}
The last inequality is a consequence of Lemma \ref{lem:l2proximal}.
It follows that:
\begin{align*}
&\BE^\mu[|w(\hat{X_n.o})(\ti{\pi}_o(X_n)\vp|\psi)-w(\hat{X_n.o})\vp(\hat{X_n^{-1}.o}).\psi(\hat{X_n.o})|]\\
&=\sum_{g\in\Ga}|w(\hat{g.o})(\ti{\pi}_o(g)\vp|\psi)-w(\hat{g.o})\vp(\hat{g^{-1}.o}).\psi(\hat{g.o})|\mu^{*n}(g)\\
&\le\sum_{g;\,|g.o|_o\ge sn}|w(\hat{g.o})(\ti{\pi}_o(g)\vp|\psi)-w(\hat{g.o})\vp(\hat{g^{-1}.o}).\psi(\hat{g.o})|\mu^{*n}(g)\\
&+\BP(|X_n|_o\le sn).\|w(\hat{g.o})(\ti{\pi}_o(g)\vp|\psi)-w(\hat{g.o})\vp(\hat{g^{-1}.o}).\psi(\hat{g.o})\|_{\ell^\infty(\Ga)}\\
&\preceq\|w\|_\infty\frac{\|\vp\|_\infty\gl(\psi)+\|\psi\|_\infty\gl(\vp)}{(1+sn)^{1/D_\e}}+(\rho\gd_\Ga^s)^n\|w\|_\infty\|\vp\|_\infty\|\psi\|_\infty
\end{align*}
which converges to $0$ when $n$ goes to infinity. 
\end{proof}

The last ingredient of Theorem \ref{thm:flipap} proof and thus Corollary \ref{cor:irr} is the uniform boundedness of the family of operators $\{\mathcal{P}_n\mid n\ge0\}$ introduced in its statement.
This is subject of the next subsection.
Assuming this fact let us conclude the proof of Theorem \ref{thm:flipap}:

\begin{proof}[Proof of Theorem \ref{thm:flipap}]
Since $(\mathcal{P}_n)_n$ are uniformly bounded it is enough to prove Theorem \ref{thm:flipap} for $\vp,\psi\in\lips(\dd X,d_{o,\e})$ which span a dense subspace in $L^2[\dd X,\nu]$.\\
Observe that:
\begin{align*}
&|(\mathcal{P}_n\vp|\psi)-\int\vp d\nu\int\psi d\nu|=|\BE^\mu[(\tilde{\pi}_o(X_n)\vp|\psi)]-\int_{\dd X}\vp\,d\nu\int_{\dd X}\psi\,d\nu|\\
&\le |\BE^\mu[(\tilde{\pi}_o(X_n)\vp|\psi)]-\BE^\mu[\vp(\hat{X_n.o}).\psi(\hat{X_n^{-1}.o})]|\\
&+|\BE^\mu[\vp(\hat{X_n.o}).\psi(\hat{X_n^{-1}.o})]-\int_{\dd X}\vp\,d\nu\int_{\dd X}\psi\,d\nu|\\
\end{align*}

On one hand Theorem \ref{thm:equi} guarantee that:
$$\BE^\mu[\vp(\hat{X_n.o}).\psi(\hat{X_n^{-1}.o})]\xrightarrow{n\rightarrow+\infty}\int_{\dd X}\vp\,d\nu\int_{\dd X}\psi\,d\nu.$$

On the other hand Proposition \ref{prop:coef} guarantees:
$$\BE^\mu[(\tilde{\pi}_o(X_n)\vp|\psi)]\xrightarrow{n\rightarrow+\infty}\BE^\mu[\vp(\hat{X_n.o}).\psi(\hat{X_n^{-1}.o})]|$$
This concludes the proof of Theorem \ref{thm:flipap}.
\end{proof}
%%%%%%%%%%%%%%%%%%%%%UNIFORM BOUNDENESS %%%%%%%%%%%%%%%%%%%%%%
%%%%%%%%%%%%%%%%%%%%%UNIFORM BOUNDENESS %%%%%%%%%%%%%%%%%%%%%%
%%%%%%%%%%%%%%%%%%%%%UNIFORM BOUNDENESS %%%%%%%%%%%%%%%%%%%%%%

\subsubsection{Uniform boundedness of the average sequence $(\mathcal{P}_n)_n$.}
We start this section with several technical preliminary results and conclude the uniform boundedness in Proposition \ref{prop:corb} and Corollary \ref{cor:corb}
\begin{lem}\label{lem:intcomp}
Let $(\dd X,d_{o,\e},\nu)$ be the Gromov boundary of $(X,d)$ together with $\nu$ a $D_\e$-Ahlfors regular probability measure.
Then 
$$\int_{\{\eta|\,(\eta|\xi)\le k\}} d_{o,\e}^{-D_\e}(\xi,\eta)d{\nu}(\eta)\asymp1+\e D_\e.k$$
for any $k\ge0$.
%$$\int_{\dd X\setminus B(\xi,T)} d_{o,\e}^{-D_\e}(\xi,\eta)d{\nu}(\eta)\asymp1+D_\e\log{T^{-1}}$$
%for all $0<T\le1$ and $\xi\in\dd X$.\\
%In particular 
 %%%to verify
\end{lem}
\begin{proof}
Using the $D_\e$-Ahlfors-regularity of $\nu$ observe that: %together with the relation $\e D_\e=\gd_\Ga$ one has:
\begin{align*}
\int_{\{\eta|\,(\eta|\xi)\le k\}} &d_{o,\e}^{-D_\e}(\xi,\eta)d{\nu}(\eta)\\
&=1+\e D_\e\int_{0}^{k} e^{\e D_\e t}.\nu_0(\{(\xi|.)>t\})dt\\
&=1+\e D_\e\int_{0}^{k}e^{\e D_\e t}.\nu_0(\{d_{o,\e}(\xi,.)<e^{-\e t }\})dt\\
&\asymp1+\e D_\e\int_{0}^{k}e^{\e D_\e t}.e^{-\e D_\e t }dt=1+\e D_\e k.
\end{align*}
%Given $0<T\le1$ one has:
%\begin{align*}
%\int_{\dd X\setminus B(\xi,T)} &d_{o,\e}^{-D_\e}(\xi,\eta)d{\nu}(\eta)\\
%&=\int_{\text{Diam($\dd X$)}^{-1}}^{T^{-1}}D_\e t^{D_\e-1}.\nu_0(\{d_{o,\e}^{-1}(\xi,.)>t\})dt\\
%&+\text{Diam($\dd X$)}^{D_\e}\\
%&=\int_{1}^{T^{-1}}D_\e t^{D_\e-1}.\nu_0(\{d_{o,\e}(\xi,.)<t^{-1}\})dt+1\\
%&\asymp\int_{1}^{T^{-1}}D_\e t^{D_\e-1}.t^{-D_e}dt=1-D_\e\log{T}.
%\end{align*}
%where the 3rd equality use the Ahlfors $D_\e$-regularity.\\
%Eventually since $(\eta|\xi)\le k$ iff $\rho_{o,\e}(\eta,\xi)\le e^{-\e k}$ observe that:
%\begin{align*}
%\int_{\{\eta|\,(\eta|\xi)\le k\}} d_{o,\e}^{-D_\e}(\xi,\eta)d{\nu}(\eta)&\asymp \int_{\dd X\setminus B(\xi,e^{-\e k})} d_{o,\e}^{-D_\e}(\xi,\eta)d{\nu}(\eta)\\
%&\asymp1+\e D_\e k=1+\gd_\Ga k
%\end{align*}
\end{proof}

\begin{lem}\label{lem:ctrinf}
Let $C\ge \gd+2$ be a positive constant.\\
There exists $C'>0$, for any pair of distinct points, $\eta_0\neq\eta'_0$, in the boundary $\dd X$ of $X$ with $C\ge(\eta_0|\eta'_0)_o+\gd+1$ the two neighborhoods of $\eta_0$ and $\eta'_0$ given respectively by:
$$U(\eta_0,C)=U=\{x\in \ol{X}\mid (x|\eta_0)_o\ge C\}$$ and $$U(\eta'_0,C)=U'=\{x\in \ol{X}\mid (x|\eta'_0)_o\ge C\}$$
satisfy $U\cap U'=\emptyset$ and for any $\eta\in \dd X\cap U$, $\eta'\in \dd X\cap U'$ the following inequality holds:
$$(g.o|\xi)_o\le \max\{(g.\eta|\xi)_o,(g.\eta'|\xi)_o\}+C'$$
for all $g\in \Ga$ and $\xi\in\dd X$.
\end{lem}
\begin{proof}
Assume there exists $x\in U\cap U'$. 
Then one has:
$$(\eta_0|\eta'_0)_o\ge\min\{(x|\eta_0)_o,(x|\eta'_0)_o\}-\gd\ge C-\gd\ge(\eta_0|\eta'_0)_o+1$$
which is a contradiction, thus $U\cap U'=\emptyset$ and in particular $\ol{X}=U^c\cup {U'}^c$.\\

Let $g\in\Ga$ and assume $g^{-1}.o\in {U'}^c$.
One might also assume that $(g.o|\xi)_o\ge (g.\eta|g.o)_o$, otherwise the inequality:
$$(g.\eta|\xi)_o\ge\min\{(g.\eta|g.o)_o,(g.o|\xi)_o\}-\gd=(g.o|\xi)_o-\gd$$
guarantees the second part of the lemma with $C'=\gd$.\\

Since $d(o,g^{-1}.o)-(o|\eta)_{g^{-1}.o}\approx(g^{-1}.o|\eta)_o$ and $g^{-1}.o\in {U'}^c$ which implies that $(g^{-1}.o|\eta)_o\le C$, one has:
\begin{align*}
(g.o|\xi)_o-C&\le |g.o|_o-C\le d(o,g^{-1}.o)-(g^{-1}.o|\eta)_o\\
&\approx(o|\eta)_{g^{-1}.o}=(g.o|g.\eta)_o\le (g.\eta|\xi)_o+\gd
\end{align*}
where the last inequality is a consequence of the hyperbolic inequality:
$$(g.\eta|\xi)_o\ge\min\{(g.o|g.\eta)_o,\,(g.o|\xi)_o\}-\gd$$
 together with our assumptions.
Eventually the lemma follows by taking $C+\gd\preceq C'$ where the extra additive constant in the choice of $C'$ only depend on the geometry on $(X,d)$.%which only depend on $\eta,\eta'\in\dd X$.
\end{proof}

%A direct consequence of Lemme \ref{lem:ctrinf} is the following:
\begin{cor}\label{cor:ctrinf}
Let $C\ge\gd+2$ be a fixed constant.
Following Lemma \ref{lem:ctrinf} notations, let $C'$ and $\eta\neq \eta'$ two distinct boundary points that belong respectively to $U$ and $U'$.\\
There exists $\gl>0$ which only depends on $C$ such that:
\begin{align*}
\BE^\mu[\tilde{\pi}_o(X_n)]&{\bf 1}_{\dd X}(\xi)\asymp\sum_{g\in\Ga}\frac{e^{\e D_\e(g.o|\xi)_o}}{1+|g.o|_o}\mu^{*n}(g)\\
&\le \gl\sum_{g\in\Ga}[f(g\eta,\xi)+f(g\eta',\xi)]\mu^{*n}(g)\\
%&+f(g\eta',\xi)f(g^{-1}\eta,\xi')+f(g\eta',\xi)f(g^{-1}\eta',\xi')]\mu^{*n}(g)
\end{align*}
for all $\xi\in\dd X$ and $n\ge0$ with
$$f:\Ga\times\dd X\times\dd X\rightarrow \BR_+;\quad (g,\eta,\xi)\mapsto \frac{d^{-D_\e}(\eta,\xi)}{1+|g.o|_o}$$
\end{cor}
\begin{proof}
Using Lemma \ref{lem:ctrinf} one has:
\begin{align*}
e^{\e D_\e(g.o|\xi)_o}&\le e^{\e D_\e[\max\{(g.\eta|\xi)_o,(g.\eta'|\xi)_o\}+C']}\\
&\le e^{\e D_\e C'}.[e^{\e D_\e(g.\eta|\xi)}+e^{\e D_\e(g.\eta'|\xi)}]\\
&\asymp e^{\e D_\e C'}.[d^{-D_\e}(g.\eta,\xi)+d^{-D_\e}(g.\eta',\xi)]
\end{align*}
for all $g\in\Ga$.
It follows that:
\begin{align*}
\frac{e^{\e D_\e(g.o|\xi)_o}}{1+|g.o|_o}&\le e^{\e D_\e C'}.[\frac{d^{-D_\e}(g.\eta,\xi)}{1+|g.o|_o}+\frac{d^{-D_\e}(g.\eta',\xi)}{1+|g.o|_o}]\\
\end{align*}
Averaging with respect to $\mu^{*n}$ proves the corollary with $\gl=e^{2\gd_\Ga C'}$.
\end{proof}

We can prove the principal ingredient of the uniform boundedness of the sequence of operators $(\mathcal{P}_n)_n$:
\begin{prop}\label{prop:corb}
The sequence of functions $(\BE^\mu[\tilde{\pi}_o(X_n)]{\bf1}_{\dd X})_n$ is uniformly bounded in $L^\infty(\dd X,\nu)$.
\end{prop}
\begin{proof}
Let $\gl>\gd_\Ga$ and $\frac{\ln E}{[\gl-\gd]}\le a$ where $\gd_\Ga=\e D_\e$ denote the critical exponent of $\Ga$ and $E=\BE^\mu[e^{\gl |X_1|_o}]$, that is finite since $\mu$ is finitely supported.

Using exponential Chebyshev inequality:
\begin{align*}
e^{\gl a(n+k)}\BP(|X_n|_o\ge a(n+k))&\le \BE[e^{\gl |X_n|_o}]\\
&\le \BE[e^{\gl |X_1|_o}]^n=E^n\\
\end{align*}
and therefore:
\begin{align*}
e^{\gd a(n+k)}\BP(|X_n|_o\ge a(n+k))&\le e^{\gd a(n+k)}e^{-\gl a(n+k)}E^n\\
%&\le e^{\gd a(n+k)}e^{-\gl a(n+k)}\BE[e^{-\gl |X_1|_o}]^n\\
&=e^{[\gd-\gl]ak+[\ln E+a\gd-a\gl]n}
\end{align*}
%Take $\gd<\gl$ and $\frac{\ln E}{[\gl-\gd]}\le a$ and observe that:
It follows
\begin{align*}
\BE[\tilde{\pi}_o(X_n){\bf 1}_{|X_n|_o\ge an}]{\bf 1}_{\dd X}(\xi)&\preceq\sum_{k=0}^{+\infty}\sum_{g\in \Ga,\,a(n+k)\le|g|_o\le a(n+k+1)}\mu^{*n}(g)\frac{e^{\e D_\e(g.o|\xi)_o}}{1+|g|_o}\\
&\le\sum_{k=0}^{+\infty}\BP(|X_n|\ge a(n+k))e^{\e D_\e a(n+k+1)}\\
&\le e^{[\ln E+a\e D_\e-a\gl]n}\sum_k e^{[\gd-\gl]ak}=\frac{e^{[\ln E+a\e D_\e-a\gl]n}}{1-e^{[\e D_\e-\gl]a}}
\end{align*}

On the other hand:
\begin{align*}
\BE[\tilde{\pi}_o(X_n){\bf 1}_{|X_n|_o\le a'n}]{\bf 1}_{\dd X}(\xi)&=\sum_{g\in \Ga,\,|g|_o\le a'n}\mu^{*n}(g)\frac{e^{\e D_\e(g.o|\xi)_o}}{1+|g|_o}\\
&\le\mu^{*n}(e)\sum_{g\in \Ga,\,|g|_o\le a'n}\frac{e^{\e D_\e(g.o|\xi)_o}}{1+|g|_o}+\\
&=\rho^ne^{\gd a'n}=e^{[\ln\rho+\e D_\e a']n}
\end{align*}
with $\frac{-\ln\rho}{\e D_\e}\le a'$.

Let us denote:
$$f_k:\dd X\times\dd X\rightarrow \BR_+;\quad (\eta,\xi)\mapsto
\begin{cases}
\frac{d^{-D_\e}(\eta,\xi)}{1+\gd_\Ga k}&\text{for $(\xi|\eta)_o\le k$}\\
\frac{\exp({\gd_\Ga k})}{1+\gd_\Ga k}&\text{otherwise}
\end{cases}$$
for $k\ge1$.
Note that Lemma \ref{lem:intcomp} implies that $\|f_k(.,\xi)\|_1\preceq1$ for all $\xi\in\dd X$ and $k\ge1$.

Take $\eta,\eta'\in\dd X$ and $U,U'\subset \dd X$ as in Lemma \ref{lem:ctrinf} and observe that:
\begin{align*}
&\BE[\tilde{\pi}_o(X_n){\bf 1}_{a'n\le|X_n|_o\le an}]{\bf 1}_{\dd X}(\xi)=\sum_{g\in \Ga,\,a'n\le|g|_o\le an}\mu^{*n}(g)\frac{e^{\e D_\e(g.o|\xi)_o}}{1+|g|_o}\\
&\le\frac{1}{1+a'n}\sum_{g\in \Ga,\,a'n\le|g|_o\le an}\mu^{*n}(g)f_n(g\eta,\xi)\\
&+\frac{1}{1+a'n}\sum_{g\in \Ga,\,a'n\le|g|_o\le an}\mu^{*n}(g)f_n(g\eta',\xi)\\
\end{align*}
\begin{align*}
&\le\frac{1}{(1+a'n)\nu(U)}\sum_{g\in \Ga}\mu^{*n}(g)\int_{\dd X}f_n(g\eta,\xi)d\eta\\
&+\frac{1}{(1+a'n)\nu(U')}\sum_{g\in \Ga}\mu^{*n}(g)\int_{\dd X}f_n(g\eta',\xi)d\eta'\\
&\le2\frac{\|f_n(.,\xi)\|_1}{(1+a'n)\min\{\nu(U),\nu(U')\}}\\
&=2\frac{1+an}{(1+a'n)\min\{\nu(U),\nu(U')\}}=2\frac{1+a}{a'\min\{\nu(U),\nu(U')\}}
\end{align*}
The proposition follows from the uniform bounds on each of these quantities.
\end{proof}

\begin{cor}\label{cor:corb}
Let $\vp,\psi\in L^2[\dd X,\nu]$ be two square integrable functions on $(\dd X,\nu)$.
Then one has:
$$|(\BE^\mu[\tilde{\pi}_o(X_n)]\vp|\psi)|\preceq \|\vp\|_2\|\psi\|_2$$
uniformly on $n$.
In particular the sequence of operator averages $(\mathcal{P}_n)_n$ is uniformly bounded.
\end{cor}
\begin{proof}
Take $\vp,\psi\in L^4(\dd X,\nu)$ and observe that:
\begin{align*}
|(\BE^\mu[\tilde{\pi}_o(X_n)]\vp|\psi)|^2&=|\int \vp(g^{-1}\xi)\psi(\xi)u_g(\xi)d\xi d\mu^{*n}(g)|^2\\
&\le|\int |\vp(g^{-1}\xi)|^2u_g(\xi)d\xi d\mu^{*n}(g)\int |\psi(\xi)|^2u_g(\xi)d\xi d\mu^{*n}(g)|^2\\
&=(\BE^\mu[\tilde{\pi}_o(X_n)]{\bf1}_{\dd X}||\psi|^2)(|\vp|^2|\BE^\mu[\tilde{\pi}_o(X_n)]{\bf1}_{\dd X})\\
&\le\|\BE^\mu[\tilde{\pi}_o(X_n)]{\bf1}_{\dd X}\|_\infty^2\|\vp\|_2^2\|\psi\|_2^2
\end{align*}
The corollary follows by density of $L^4(\dd X,\nu)$ inside of $L^2[\dd X,\nu]$.
\end{proof}
%%%%%%%%%%%%%%%%%%%%%%%%% JUST ADDED %%%%%%%%%%%%%%%%%%%%%%%%%%%%%%%
%%%%%%%%%%%%%%%%%%%%%%%%% JUST ADDED %%%%%%%%%%%%%%%%%%%%%%%%%%%%%%%

%\subsubsection{Irreducibility: Theorem \ref{thm:flipap} end of the proof}

%%%%%%%%%%%%% RECURRENCE PART.%%%%%%%%%%%%%%%%%%%%%%%%%%%%%
%\section{Recurrent sets and subgroups}
%%%%%%%%%%%%%%%%%SUBSECTION EQUIDISTRIBUTION REC%%%%%%%%%%%%%
%%%%%%%%%%%%%%%%%%%%%%%%%%%%%%%%%%%%%%%%%%%%%%%%%%%%%%%%%%%%%%%%%
%%%%%%%%%%%%%%%%%%%%%%%%%%%%%%%%%%%%%%%%%%%%%%%%%%%%%%%%%%%%%%%%%
%%%%%%%%%%%%%%%%%%%%%%%%%%%%%%%%%%%%%%%%%%%%%%%%%%%%%%%%%%%%%%%%%
\nocite{*}
\bibliographystyle{plain}
\bibliography{biblio_rec}
\end{document}